\documentclass[12pt]{amsart}
\usepackage[utf8]{inputenc}
\usepackage{amsfonts}
\usepackage{amssymb}
\usepackage{amsmath}
\usepackage[english]{babel}
\usepackage{amsthm}
\theoremstyle{plain}
\usepackage{enumitem}

\newtheorem{mainthm}{Theorem}

\newtheorem{theorem}{Theorem}[section]
\newtheorem{corollary}[theorem]{Corollary}
\newtheorem{lemma}[theorem]{Lemma}
\newtheorem{proposition}[theorem]{Proposition}
\theoremstyle{definition}

\newtheorem{remark}[theorem]{Remark}

\usepackage{geometry}
\numberwithin{equation}{section} 

\usepackage{lipsum}

\newcommand\blfootnote[1]{%
  \begingroup
  \renewcommand\thefootnote{}\footnote{#1}%
  \addtocounter{footnote}{-1}%
  \endgroup
}

\usepackage[colorlinks, citecolor=blue,pagebackref,hypertexnames=false]{hyperref}

\newcounter{comcount}

\begin{document}

\title[Reverse Riesz inequality on metric cone]{On the Reverse Inequality of Riesz transform on metric cone with potential} 
\date{}
\author{DANGYANG HE}  %\sffamily
\address{Department of Mathematics, Sun Yat-sen University, Guangzhou, China, Formerly School of Mathematics and Physics, Macquarie University, Sydney, Australia}
\email{hedy28@mail.sysu.edu.cn}

% abstract
%{\noindent\small{\bf Abstract:}
\begin{abstract}
Let $M=(0,\infty)_r\times Y$ be a $d$-dimensional ($d\ge3$) metric cone with metric $g=dr^2+r^2h$, where $(Y,h)$ is a closed Riemannian manifold. Consider the Schrödinger operator $H=\Delta+V_0/r^2$, with $V_0\in C^\infty(Y)$ satisfying the positivity condition: $\Delta_Y+V_0+(d-2)^2/4>0$.

First, we complement previous results by proving Lorentz–type endpoint estimates for the Riesz transform $\nabla H^{-1/2}$: it is of restricted weak type at both endpoints of its $L^p$-boundedness range. Second, we establish the sharp reverse inequality
\[
\|H^{1/2}f\|_{p} \lesssim \|\nabla f\|_{p}+\bigl\|\tfrac{f}{r}\bigr\|_{p}
\]
if and only if
\[
\frac{d}{\min\left(\frac{d+4}{2}+\mu_0,\, d\right)} < p < \frac{d}{\max\left(\frac{d-2}{2}-\mu_0,\, 0\right)},
\]
where $\mu_0>0$ is the square root of the smallest eigenvalue of $\Delta_Y+V_0+(d-2)^2/4$. 

% We also provide the corresponding Lorentz–type endpoint estimates at these critical exponents.
\end{abstract}
%\vspace{1ex}

\maketitle

% main body
\tableofcontents

\blfootnote{$\textit{2020 Mathematics Subject Classification.}$ 42B37, 58J05}

\blfootnote{$\textit{Keywords and Phrases.}$ Riesz transform, reverse Riesz inequality, Schrödinger operator, endpoint estimate, metric cone.}

\section{Introduction}
Let $M$ be a complete Riemannian manifold. Denote by $\nabla$ and $\Delta$ the usual Riemannian gradient and Laplace--Beltrami operator, respectively. The Riesz transform is given by
\begin{equation*}
    R = \nabla \Delta^{-1/2}.
\end{equation*}

First posed by Strichartz \cite{Str}, the question of the $L^p$-boundedness of $R$ has been investigated for decades. In particular, we say that the Riesz transform is bounded on $L^p$ (or $L^p$-bounded) if $\|R\|_{p\to p}\le C$, or equivalently if
\begin{equation}\label{Rp}\tag{$\textrm{R}_p$}
    \| \nabla f \|_p \le C \| \Delta^{1/2} f \|_p, \quad \forall f\in C_c^\infty(M).
\end{equation}
We say that the reverse Riesz inequality holds if
\begin{equation}\label{RRp}\tag{$\textrm{RR}_p$}
    \| \Delta^{1/2} f \|_p \le C \| \nabla f \|_p, \quad \forall f\in C_c^\infty(M).
\end{equation}

These problems turn out to be extremely difficult because of their strong dependence on the geometry of the underlying structure. Unlike in the Euclidean case, the Riesz transform problem \eqref{Rp} is usually divided into two parts. For $p<2$, the question is comparatively well understood. Indeed, under very mild assumptions, Coulhon and Duong \cite{CoDu} show that $R$ is $L^p$-bounded for all $1<p\le 2$ and is of weak type $(1,1)$. For $p>2$, gradient estimates (for instance for the heat kernel or for local harmonic functions) turn out to be particularly significant.

Recall the volume doubling condition
\begin{equation}\label{doubling}
V(x,2r) \le C V(x,r),\quad \forall x\in M,\quad \forall r>0,
\end{equation}
where $V(x,r)$ denotes the volume of the ball $B(x,r)$, and the $L^q$-Poincaré inequality
\begin{equation}\label{Pq}\tag{$\textrm{P}_q$}
    \int_B |f-f_B|^q \, d\mu \le C r^q \int_B |\nabla f|^q \, d\mu,\quad \forall f\in C^\infty(B),\quad \forall B\subset M.
\end{equation}

According to \cite{ACDH}, under the assumptions \eqref{doubling} and $(\textrm{P}_2)$, the boundedness of the Riesz transform \eqref{Rp} is equivalent to the following heat kernel regularity:
\[
    \sup_{t>0} \bigl\| \sqrt{t} \,\nabla e^{-t\Delta}\bigr\|_{p\to p} \le C.
\]
Subsequently, Coulhon--Jiang--Koskela--Sikora and Jiang added a so-called reverse Hölder condition to extend the above equivalence under even weaker assumptions; see \cite{CJKS} and \cite{Jiang}, respectively.

However, on general manifolds, both conditions \eqref{doubling} and \eqref{Pq} can be violated quite easily. For example, on Lie groups with exponential growth, \eqref{doubling} fails, and on the connected sum of two Euclidean spaces, $(\textrm{P}_2)$ fails to hold. Nevertheless, in the Lie group with exponential volume growth, \cite{Hebisch} (see also \cite{SV}) implies that the first-order Riesz transform is still bounded in $L^p$ for all $1<p<\infty$. For the case $\mathbb{R}^n \# \mathbb{R}^n$, where the Poincaré inequality fails, the result of \cite{CCH} (see also \cite{CoDu}) shows that \eqref{Rp} holds if and only if $1<p<n$. For more studies of the Riesz transform on perturbed manifolds, we refer the reader to \cite{HS,HNS,Carron} and the references therein. 

In this article, however, we will focus on the metric cone. Let $(Y,h)$ be a smooth closed Riemannian manifold. The metric cone is defined to be the product $M = (0,\infty)_r \times Y$ with metric $g = dr^2 + r^2 h$. In particular, $\mathbb{R}^n \setminus \{0\}$ is a metric cone of dimension $n$ with cross section $\mathbb{S}^{n-1}$. It is known that if $Y$ is connected (so that $M$ has only one end), then $M$ satisfies \eqref{doubling} and $(\textrm{P}_1)$ (see \cite{Coulhon-Li, GrSa}), and its corresponding heat kernel estimates can be found in \cite{HQLi2, Coulhon-Li}. 

The question regarding \eqref{Rp} was answered by Li in the following theorem.

\begin{theorem}\cite[Théorème~I]{HQLi}\label{thmL}
Let $d\ge 3$ and let $M$ be a metric cone of dimension $d$ with cross section $Y$. Then \eqref{Rp} holds if and only if 
\begin{equation}
    1< p < \frac{d}{\max\left( \frac{d}{2}-\mu_1, 0 \right)},
\end{equation}
where $\mu_1>0$ is the square root of the second smallest eigenvalue of the operator $\Delta_Y + \left( \frac{d-2}{2} \right)^2$.
\end{theorem}

The above result was generalized by Hassell and Lin \cite{HL} by adding an inverse-square potential. Let $V_0: Y \to \mathbb{C}$ be a smooth function such that
\begin{equation}\label{positivity}
    \Delta_Y + V_0(y) + \left( \frac{d-2}{2} \right)^2 > 0.
\end{equation}
Note that this condition allows $V_0$ to be negative (for example, $V_0$ may be any constant larger than $-\frac{(d-2)^2}{4}$). Consider the Schrödinger operator
\[
    H = \Delta + \frac{V_0}{r^2}.
\]
Observe that $H$ is self-adjoint and positive. We can then define the associated Riesz transform by
\begin{equation*}
    \nabla H^{-1/2} = \nabla \left( \Delta + \frac{V_0}{r^2} \right)^{-1/2}.
\end{equation*}

\begin{theorem}\cite[Theorem~1.3]{HL}\label{thmHL}
Let $d\ge 3$ and let $M$ be a metric cone of dimension $d$ with cross section $Y$. Let $V_0$ be a smooth function defined on $Y$ that satisfies \eqref{positivity} and $V_0 \not\equiv 0$. Then the Riesz transform with potential $V = \frac{V_0}{r^2}$ is bounded on $L^p$ if and only if $p$ lies in the interval
\begin{equation}\label{interval}
    \left( \frac{d}{\min \left( \frac{d+2}{2} + \mu_0, d \right)},  \frac{d}{\max \left( \frac{d}{2} - \mu_0, 0 \right)}    \right),
\end{equation}
where $\mu_0$ is the square root of the smallest eigenvalue of the operator $\Delta_Y + V_0 + \left( \frac{d-2}{2} \right)^2$.

That is, the operator $\nabla H^{-1/2} = \nabla \left( \Delta + V \right)^{-1/2}$ is $L^p$-bounded if and only if $p$ belongs to the interval \eqref{interval}.
\end{theorem}

The “only if” part of the above result, i.e., the range of unboundedness of $\nabla H^{-1/2}$, inherits ideas from Guillarmou and Hassell \cite{GH1}, who consider the Riesz transform for Schrödinger operators with no zero modes and no zero resonances on asymptotically conic manifolds and show that $\nabla H^{-1/2}$ cannot be bounded for $p$ outside the range \eqref{interval}. The case with zero modes or zero resonances is treated in \cite{GH2}.

The study of the Riesz transform can also be carried out from the point of view of Ricci curvature. Indeed, in \cite{bak87}, Bakry proved that on manifolds with non-negative Ricci curvature, \eqref{Rp} holds for all $1<p<\infty$. In fact, manifolds with conical ends and the connected sum of finitely many Euclidean spaces both satisfy a quadratic lower bound for the Ricci curvature:
\begin{equation}\label{QD}
    \mathrm{Ric}_x \ge - \frac{\epsilon^2}{\left( 1 + r(x) \right)^2} g_x,\quad \forall x\in M,
\end{equation}
where $r(x) = d(o,x)$ for some fixed point $o\in M$. In \cite{Carron}, Carron studied manifolds satisfying the condition \eqref{QD}, and gave a characterization of \eqref{Rp} in terms of a lower bound for volume growth, which we refer to as the reverse doubling condition:
\begin{equation}\label{RD}
    \frac{V(x,R)}{V(x,r)} \ge C \left(\frac{R}{r} \right)^\nu,\quad \nu>2.
\end{equation}
In particular, his result can be applied to the cases studied in \cite{HQLi,CCH}. Precisely, on a metric cone $M$, we have volume estimate $V(x,r) \sim r^d$ (hence \eqref{RD} holds with $\nu=d$). Moreover, it is clear that the cone metric implies \eqref{QD}. Then, \cite[Theorem~A]{Carron} guarantees that \eqref{Rp} holds on $M$ for $1<p<d$. Note that this result is sharp if $M$ has multiple ends.

In this note, we continue the study of the Riesz transform for Schrödinger operators on metric cones. We also consider a slightly modified, “left-lifted’’ version of the Riesz transform. For $f\in C^\infty(M)$, define
\[
    \nabla_H f := (\nabla f, f/r),
\]
viewed as a section of the product bundle $\Lambda^1(M)\times\Lambda^0(M)$, and set
\[
    |\nabla_H f| := \sqrt{|\nabla f|^2 + \left|f/r\right|^2}.
\]
For $1\le p\le \infty$, we define the seminorm
\begin{equation}\label{d_H}
    \|\nabla_H f\|_p^p := \int_M  \left(|\nabla f|^2 + \left|f/r\right|^2\right)^{\frac{p}{2}} d\mu  \sim\|\nabla f\|_p^p + \left\|f/r\right\|_p^p,
\end{equation}
and we say that the Riesz transform for the Schrödinger operator is bounded on $L^p$ if
\begin{equation}\label{RVp}\tag{$\textrm{RV}_p$}
    \|\nabla_H f\|_p  \le C \| H^{1/2} f \|_p, \quad \forall f\in C_c^\infty(M).
\end{equation}
The reason for introducing this notion is to study the following reverse inequality in a more thorough way:
\begin{equation}\label{RRVp}\tag{$\textrm{RRV}_p$}
    \| H^{1/2} f \|_p \le C \|\nabla_H f\|_p, \quad \forall f\in C_c^\infty(M).
\end{equation}

\subsection{Main results}

Before presenting our main results, let us first introduce some notation. Throughout the article, we write
\begin{align}\label{p0p1}
    p_0:= \frac{d}{\min\left( \frac{d+2}{2}+\mu_0, d \right)}, \quad 
    p_1:= \frac{d}{\max \left( \frac{d}{2} - \mu_0, 0 \right)},
\end{align}
where $\mu_0>0$ is the square root of the smallest eigenvalue of the operator
\[
    \Delta_Y + V_0(y) + \left(\frac{d-2}{2}\right)^2.
\]
Thus
\begin{align}\label{p0'p1'}
    p_0' = \frac{d}{\max\left(\frac{d-2}{2}-\mu_0, 0\right)},\quad 
    p_1' = \frac{d}{\min \left(\frac{d}{2}+\mu_0, d\right)}.
\end{align}

We recall that a function $f$ belongs to the Lorentz space $L^{p,q}$ ($0<p,q\le \infty$) if the quasi-norm
\begin{equation*} 
     \|f\|_{(p,q)}=  
    \begin{cases}
    \left( \displaystyle\int_0^{\infty} \left(t^{1/p} f^*(t)\right)^q \frac{dt}{t}   \right)^{1/q}, & q<\infty,\\[1em]
     \displaystyle\sup_{t>0} t^{1/p}f^*(t) = \sup_{\lambda>0} \lambda\, d_f(\lambda)^{1/p}, & q=\infty,
    \end{cases}
\end{equation*}
is finite, where $f^*$ denotes the decreasing rearrangement of $f$,
\[
    f^*(t) = \inf\{\lambda>0 : d_f(\lambda)\le t\},
\]
and $d_f$ is the usual distribution function,
\[
    d_f(\lambda) = \mu\bigl(\{x : |f(x)| >\lambda  \}\bigr).
\]

As observed in Theorem~\ref{thmL} and Theorem~\ref{thmHL}, \eqref{Rp} and \eqref{RVp} hold only on a finite interval of exponents. Therefore, it is natural to ask whether there exists a suitable endpoint estimate. From \cite{HQLi}, it is known that the Riesz transform $R$ on a metric cone is not even of weak type $(p,p)$ at the endpoint. Hence, as a natural candidate, one expects $R$ to be of restricted weak type bounded at the endpoints.

Our first result confirms this expectation.

% \begin{mainthm}\label{thm1}
% Let $d\ge 3$ and let $M = (0,\infty)_r \times Y$ be a metric cone of dimension $d$ with cross section $Y$. Suppose $V_0$ is a smooth function defined on $Y$ that satisfies \eqref{positivity}. Let $V = V_0/r^2$ be the corresponding inverse-square potential and let $H = \Delta + V$ be the associated Schrödinger operator. Then the Riesz transform $R = \nabla H^{-1/2}$ is of restricted weak type $(p_0,p_0)$. Moreover, if $p_1<\infty$, then $R$ is also of restricted weak type $(p_1,p_1)$, i.e.,
% \begin{equation*}
%     \| Rf\|_{(p_i, \infty)} \le C \| f \|_{(p_i, 1)}
% \end{equation*}
% for $i=0,1$, where $p_i$ are defined as in \eqref{p0p1}.

% In addition, this estimate is sharp in the sense that $R$ cannot be bounded from $L^{p_i,q}$ to $L^{p_i,\infty}$ for any $1<q\le \infty$.
% \end{mainthm}

\begin{mainthm}\label{thm1}
Let $d\ge 3$ and let $M = (0,\infty)_r \times Y$ be a metric cone of dimension $d$ with cross-section $Y$. Suppose $V_0$ is a smooth function on $Y$ satisfying \eqref{positivity}. Let $V = V_0/r^2$ be the corresponding inverse-square potential and let $H = \Delta + V$ be the associated Schr\"odinger operator. Then the Riesz transform $R = \nabla H^{-1/2}$ is of restricted weak type $(p_0,p_0)$. Moreover, if $p_1<\infty$, then $R$ is also of restricted weak type $(p_1,p_1)$, that is,
\begin{equation*}
    \| Rf\|_{(p_i,\infty)} \le C \| f \|_{(p_i,1)}
\end{equation*}
for $i=0,1$, where $p_i$ are defined in \eqref{p0p1}.

In addition, this estimate is sharp in the following sense: for $i=0,1$, the operator $R$ cannot be bounded from $L^{p_i,q}$ to $L^{p_i,\infty}$ for any $1<q\le \infty$.
\end{mainthm}

Note that the above result is consistent with the result in \cite{He1}, which in fact provides an even sharper Lorentz-type estimate in the framework of manifolds with ends.

Next, we investigate the reverse inequality \eqref{RRVp}, which is the main motivation of this article. The reason for adding the term $\left\| \frac{f}{r} \right\|_p$ on the right-hand side of \eqref{RRVp} is the identity:
\begin{equation*}
 \| H^{1/2}f\|_2^2 
 = \left\langle H^{1/2}f, H^{1/2}f \right\rangle 
 = \left\langle Hf, f \right\rangle 
 = \|\nabla f\|_2^2 + \left\langle \frac{V_0 f}{r}, \frac{f}{r} \right \rangle.
\end{equation*}
In fact, a sharper version,
\begin{equation*}
    \|H^{1/2}f\|_p \le C \|\nabla f\|_p,
\end{equation*}
can be obtained easily by using a Hardy-type inequality at the cost of losing information at a single point (see Corollary~\ref{cor1} below).

It is well known that, on a complete Riemannian manifold, a standard duality argument shows that \eqref{Rp} implies $(\textrm{RR}_{p'})$, whereas the converse is not true in general; see \cite{CD}. We show in Section~\ref{sec2} that this duality property still holds in the setting of metric cones with potential: that is, we prove that \eqref{RVp} implies $(\textrm{RRV}_{p'})$ for some range of $p$ (see Lemma~\ref{lemma2} below). Thus, with the (formally) expected equivalence of \eqref{RVp} and $(\textrm{RRV}_{p'})$ in mind, one might conjecture that the precise range for the boundedness of the reverse inequality should be
\begin{equation*}
    p_1' < p < p_0'.
\end{equation*}
However, this presumed equivalence between \eqref{Rp} and $(\textrm{RR}_{p'})$ has been shown to fail in many settings. For manifolds satisfying \eqref{QD}, see \cite{HeQD}; for the connected sum, see \cite{He2}; and for the Dirichlet Laplacian outside a bounded domain, see \cite{JL,KVZ}.

On a metric cone, suppose that $M$ has only one end (i.e.\ the cross section $Y$ is connected). Then the question of \eqref{RRp} falls within the scope of \cite[Theorem~0.7]{AC}. Indeed, under the assumptions \eqref{doubling} and \eqref{Pq} (for all $q\ge 1$), \eqref{RRp} holds for all $1<p<\infty$. When $Y$ is disconnected (so $M$ has multiple ends), the Poincaré inequality may fail (cf.\ \cite{GrSa}) and the general method of \cite{AC} does not apply. Note that $M$ satisfies \eqref{QD}. A recent article \cite{HeQD} shows that \eqref{RRp} is actually equivalent to Hardy's inequality for $p$ small. As a result, \eqref{RRp} holds on $M$ for all $1<p<\infty$ (see Lemma~\ref{Hardy} below). 

The problem becomes even more subtle once we add a possibly negative potential to the Laplacian, even if we only consider inverse-square type potentials. Indeed, in the discussion of the previous paragraph and in the special cases considered in \cite{He2,KVZ}, a Gaussian-type upper bound for the heat kernel is always (at least implicitly) assumed. In contrast, for the Schr\"odinger operator on a metric cone, there is no hope of obtaining such a bound for the heat kernel $e^{-tH}(z,z')$, and a recent article \cite{HZ} confirms this expectation.

% Moreover, this fact also leads to another interesting phenomenon: there appears a lower threshold for the $L^p$-boundedness of the Riesz transform. Reformulated in terms of the reverse inequality, this gives an upper threshold for the reverse Riesz transform. It has been verified in \cite{He2,HeQD} that it is possible to relax the lower threshold for the reverse inequality (that is, the upper threshold for the validity of \eqref{Rp}) by using the so-called \emph{harmonic annihilation method}. One might therefore hope that the upper threshold could also be relaxed by further developing this technique. Unfortunately, the next theorem shows that, in the setting of a metric cone with potential, this upper threshold is in fact sharp and cannot be improved. It turns out that those two methods introduced in Subsection~\ref{sec1.3} below: either eliminating the harmonic leading coefficient from outside (through bilinear form) or from inside (by writing $H^{1/2}f = H^{-1/2} Hf$), produce exactly the same range of $p$ for \eqref{RRVp} to hold.

Moreover, this fact also leads to another interesting phenomenon: there is a lower threshold for the $L^p$-boundedness of the Riesz transform. Reformulated in terms of the reverse inequality, this yields an upper threshold for the reverse Riesz transform. It has been verified in \cite{He2,HeQD} that it is possible to relax the lower threshold for the reverse inequality (that is, the upper threshold for the validity of \eqref{Rp}) by using the so-called \emph{harmonic annihilation method}. One might therefore hope that the upper threshold could also be relaxed by further developing this technique. Unfortunately, the next theorem shows that, in the setting of a metric cone with potential, this upper threshold is in fact sharp and cannot be improved. It turns out that the two methods introduced in Subsection~\ref{sec1.3} below—either eliminating the harmonic leading coefficient from outside (via a bilinear form) or from inside (by writing $H^{1/2}f = H^{-1/2}Hf$)—yield exactly the same range of $p$ for which \eqref{RRVp} holds.

Nevertheless, in this article, we inherit the method of \cite{HL} and examine the explicit formula for the off-diagonal part of the kernel $(H+1)^{-1}(z,z')$. By applying the \emph{harmonic annihilation method}, we obtain the following result.

% \begin{mainthm}\label{thm2}
% Let $d\ge 3$ and let $M = (0,\infty)_r \times Y$ be a metric cone of dimension $d$ with cross section $Y$. Suppose $V_0$ is a smooth function defined on $Y$ that satisfies \eqref{positivity}. Let $V = V_0/r^2$ be the corresponding inverse-square potential and let $H = \Delta + V$ be the associated Schrödinger operator. Suppose $1<p<\infty$. Then the reverse inequality \eqref{RRVp} holds if and only if 
% \begin{align}\label{interval2}
%     \frac{d}{\min\left( \frac{d+4}{2}+\mu_0, d\right)} < p < p_0' = \frac{d}{\max\left(\frac{d-2}{2}-\mu_0, 0 \right)}.
% \end{align}
% That is,
% \begin{align*}
%     \|H^{1/2}f\|_p \le C \| \nabla_H f\|_p
% \end{align*}
% holds if and only if $p$ belongs to the range \eqref{interval2}.

% In addition, the following endpoint estimate:
% \begin{equation}\label{RRend}
%     \| H^{1/2}f\|_{(p,\infty)} \le C \| \nabla_H f\|_{(p,q)},
% \end{equation}
% holds for $q=1$ and fails for any $1<q\le \infty$, where $p=\frac{d}{\min\left( \frac{d+4}{2}+\mu_0, d\right)}$ or $p=p_0'$ if $p_0'<\infty$. 
% \end{mainthm}

\begin{mainthm}\label{thm2}
Let $d\ge 3$ and let $M = (0,\infty)_r \times Y$ be a metric cone of dimension $d$ with cross-section $Y$. Suppose $V_0$ is a smooth function on $Y$ satisfying \eqref{positivity}. Let $V = V_0/r^2$ be the corresponding inverse-square potential and let $H = \Delta + V$ be the associated Schr\"odinger operator. Suppose $1<p<\infty$. Then the reverse inequality \eqref{RRVp} holds if and only if 
\begin{align}\label{interval2}
    \frac{d}{\min\left( \frac{d+4}{2}+\mu_0, d\right)} < p < p_0' = \frac{d}{\max\left(\frac{d-2}{2}-\mu_0, 0 \right)}.
\end{align}
Equivalently,
\begin{align*}
    \|H^{1/2}f\|_p \le C \| \nabla_H f\|_p
\end{align*}
holds if and only if $p$ belongs to the range \eqref{interval2}.

In addition, the following endpoint estimate:
\begin{equation}\label{RRend}
    \| H^{1/2}f\|_{(p,\infty)} \le C \| \nabla_H f\|_{(p,q)},
\end{equation}
holds for $q=1$ and fails for every $1<q\le \infty$, where $p=\frac{d}{\min\left( \frac{d+4}{2}+\mu_0, d\right)}$ or $p=p_0'$ if $p_0'<\infty$. 
\end{mainthm}

\begin{remark}
It is clear that
\begin{equation*}
    \frac{d}{\min\left( \frac{d+4}{2}+\mu_0, d\right)} \le p_1' = \frac{d}{\min\left( \frac{d}{2}+\mu_0, d\right)}.
\end{equation*}
In other words, the lower threshold in the “presumed dual range’’ for \eqref{RRVp} has been relaxed. Note that if $V_0 \ge 0$, then $\mu_0 > \frac{d-2}{2}$, and hence $p_0 = 1$, so that $p_0'=\infty$. Moreover, in this case the lower threshold in \eqref{interval2} equals $1$, which shows that \eqref{RRVp} holds for all $1<p<\infty$. 

On the other hand, if $V_0$ is negative (so $0<\mu_0\le \frac{d-2}{2}$), then $p_0$ always lies in $(1,2)$, and the upper threshold, i.e., $p_0'$, in \eqref{interval2} remains an insurmountable barrier. In addition, if $\frac{d-4}{2}<\mu_0\le \frac{d-2}{2}$, then the lower threshold in \eqref{interval2} is again $1$, whereas for sufficiently negative potentials (i.e., for $0<\mu_o\le \frac{d-4}{2}$) the lower threshold also lies strictly between $1$ and $2$. See table below.

\begin{table}[htbp]
\centering
\caption{Range of $p$ in \eqref{RRVp} depending on $\mu_0$.}
\begin{tabular}{ccc}
\hline
Assumption on $\mu_0$ & Sign / size of $V_0$ & Admissible $p$ in \eqref{RRVp} \\
\hline
$\mu_0 > \dfrac{d-2}{2}$ 
  & $V_0 \ge 0$ 
  & $1 < p < \infty$ \\[0.6em]

$\dfrac{d-4}{2} < \mu_0 \le \dfrac{d-2}{2}$ 
  & $V_0 < 0$ (mildly negative) 
  & $1 < p < \dfrac{2d}{d-2-2\mu_0}$ \\[0.9em]

$0 < \mu_0 \le \dfrac{d-4}{2}$ 
  & $V_0$ sufficiently negative 
  & \(
     \dfrac{2d}{d+4+2\mu_0} < 
     p < 
      \dfrac{2d}{d-2-2\mu_0}
    \) \\[0.4em]
\hline
\end{tabular}
\end{table}

\end{remark}

% \begin{remark}
% In fact, by the explicit counterexamples constructed in the proof (see Lemma \ref{lowerthrshold} and Lemma~\ref{upperthrshold} below), the endpoint estimate \eqref{RRend} is sharp in the Lorentz sense, i.e., the estimate:
% \begin{equation*}
%     \|H^{1/2}f\|_{(p,\infty)} \le C \| \nabla_H f\|_{(p,q)}
% \end{equation*}
% holds for $q=1$ but fails for every $q > 1$, where $p=p_0'$ or $p=\frac{d}{\frac{d+4}{2}+\mu_0}$.
% \end{remark}

As a consequence of Hardy’s inequality (see Lemma~\ref{Hardy} below), the above result can be “improved’’ to the following.

\begin{corollary}\label{cor1}
Let $p\in (1,\infty) \setminus \{d\}$. Under the assumptions of Theorem~\ref{thm2}, the reverse inequality
\begin{equation*}
    \|H^{1/2}f\|_p \le C \|\nabla f\|_p,\quad \forall f\in C_c^\infty(M),
\end{equation*}
holds if and only if $p$ lies in the range \eqref{interval2}.
\end{corollary}

Meanwhile, we have the following consequence for \eqref{RVp}.

\begin{corollary}\label{cor2}
Under the assumptions of Theorem~\ref{thmHL}, \eqref{RVp} holds if and only if $p$ belongs to the range \eqref{interval}.
\end{corollary}

Immediately, we get the following norm equivalent property.

\begin{corollary}\label{Ep}
Under the assumptions of Theorem~\ref{thmHL}, the inequality:
\begin{equation}
    C_1 \| H^{1/2} f \|_p \le  \| \nabla_H f \|_p \le C_2 \| H^{1/2} f \|_p, \quad \forall f\in C_c^\infty(M),
\end{equation}
holds if and only if 
\begin{equation*}
    \max \left( \frac{d}{\min \left(\frac{d+4}{2}+\mu_0, d\right)}, p_0 \right) < p < \min \left( p_0', p_1 \right).
\end{equation*}
\end{corollary}

\begin{proof}[Proof of Corollary~\ref{Ep}]
This is a direct consequence of Corollary~\ref{cor2} and Theorem~\ref{thm2}.
\end{proof}

By applying our theorem to the special case where $V_0 = c \ne 0$, we obtain the following result.

\begin{corollary}\label{cor3}
Let $d\ge 3$ and let $M = (0,\infty)_r \times Y$ be a metric cone of dimension $d$ with cross section $Y$. Suppose $c>- \left(\frac{d-2}{2}\right)^2$ and $c\ne 0$. Then the reverse inequality \eqref{RRVp} for the Schrödinger operator $H = \Delta+ \frac{c}{r^2}$ holds if and only if
\begin{equation*}
    \frac{2d}{\min \left( d+4+\sqrt{(d-2)^2+4c}, 2d \right)} < p < \frac{2d}{\max\left( d-2-\sqrt{(d-2)^2 + 4c}, 0 \right)}.
\end{equation*}
In addition, a similar endpoint estimate as in Theorem~\ref{thm2} holds.
\end{corollary}

\subsection{Motivation}\label{sec1.3}

In this subsection, we discuss the idea behind the proof of Theorem~\ref{thm2}, namely the
so–called \emph{harmonic annihilation method}; see also
\cite{He2,HeQD}.  From \cite{CCH}, it is known that on a
compactification $\overline{M}$ of the connected sum
$M=\mathbb{R}^n \# \mathbb{R}^n$ ($n\ge 3$), the Riesz potential has
an asymptotic expansion in the off–diagonal region
\begin{equation}\label{asy1}
    \Delta^{-1/2}(x,y) \sim \sum_{j=n-1}^\infty a_j(x)\, |y|^{-j},
    \qquad x\in \overline{M},\quad y \to \partial \overline{M},
\end{equation}
where the leading coefficient $a_{n-1}$ is a bounded, non-trivial
harmonic function on $\overline{M}$.  By the maximum principle,
$a_{n-1}$ is non-constant, so in particular its gradient does not
vanish identically.  Consequently, the leading term decays like
$|y|^{-(n-1)}$ in the $y$–variable, which prevents the Riesz transform
from being $L^p$–bounded for $p\ge n$.

Although the expansion \eqref{asy1} has only been verified in the setting of $\mathbb{R}^n \# \mathbb{R}^n$, it is natural to expect an analogous statement on a broader class of manifolds. In fact, on a compactification $\overline{M}$ of the metric
cone $M=(0,\infty)_r \times Y$, the Riesz potential associated with
$H=\Delta+V$ admits an expansion of the form (we write $z=(r,y)$ and $z'=(r',y')$)
\begin{equation}\label{asy2}
    H^{-1/2}(z,z') \sim \sum_{j=0}^\infty \sum_{n=0}^\infty \mathcal{A}_{n,j}(z,y')\, \rho(z')^{\phi(j,n)},
    \qquad \mathcal{A}_{n,j}\in C^\infty(\overline{M} \times \partial \overline{M}), \quad z\in \overline{M},\quad z'\to \partial \overline{M},
\end{equation}
for some non-decreasing function $\phi$ (in both $j$ and $n$), where $\rho$ is a boundary defining function near this regime. In particular, for each $j\ge 0$, coefficient $\mathcal{A}_{0,j}$ is $H$–harmonic, i.e.\ $H\mathcal{A}_{0,j}=0$.

Next, for $f,g\in C_c^\infty(\overline{M})$, consider the bilinear form
\begin{equation*}
    \left\langle H^{1/2}f, g \right\rangle
    = \left\langle Hf, H^{-1/2}g \right\rangle
    = \left\langle \nabla f, \nabla H^{-1/2}g \right\rangle
      + \left\langle r^{-1} V_0 f,\, r^{-1} H^{-1/2}g \right\rangle.
\end{equation*}
Assume that $g$ is supported near the boundary $\partial\overline{M}$.
By \eqref{asy2} and integration by parts, the above bilinear form
admits an asymptotic expansion 
\begin{align}\label{asy3}
    \Big\langle f,\,
       \int \sum_{j,n=0}^\infty 
       H \mathcal{A}_j(z,y')\, \rho(z')^{\phi(j,n)} g(z')\, dz' \Big\rangle
    &= \Big\langle \frac{f}{r},\,
       r \int \sum_{j=0}^\infty \sum_{n=1}^\infty
       H \mathcal{A}_j(z,y')\, \rho(z')^{\phi(j,n)} g(z')\, dz' \Big\rangle,
\end{align}
since $\mathcal{A}_{0,j}$ is $H$–harmonic. The new leading term has better decay in
the $z'$–variable (since $\phi$ is non-decreasing and $\rho(z')$ can be taken as $1/r'$ near this boundary), and one may therefore expect that
\eqref{asy3} is $O(\|\nabla_H f\|_p \|g\|_{p'})$ for some range of $p$
strictly larger than the ``dual range'' $(p_1',p_0')$ associated with
the Riesz transform.

The above heuristic has already been mentioned and used in
\cite{He2,HeQD}. Motivated by the appearance of a lower threshold in
the range of $p$ for which the Riesz transform is bounded, we now look
at the reverse Riesz inequality from a different point of view. By
symmetry, \eqref{asy2} also suggests that 
\begin{equation*}
    H^{-1/2}(z,z') \sim \sum_{j=0}^\infty \sum_{n=0}^\infty \rho(z)^{\phi(j,n)} \mathcal{A}_{n,j}(z',y),
    \qquad z\to \partial \overline{M},\quad z'\in \overline{M},
\end{equation*}
where near this regime we may define $\rho(z)=1/r$. Proceeding with $z\to \partial\overline{M}$, we deduce (using the locality of $H$) that
\begin{align}\label{asy4}
    H^{1/2}f(z)
    = H^{-1/2}Hf(z)
    &\sim \int_{\overline{M}} \sum_{j,n=0}^\infty
         \rho(z)^{\phi(j,n)} \mathcal{A}_{n,j}(z',y)\, Hf(z')\, dz'\\ \nonumber
    &\sim \int_{\overline{M}} \sum_{j=0}^\infty \sum_{n=1}^\infty
         \rho(z)^{\phi(j,n)} r' H \mathcal{A}_{n,j}(z',y)\, \frac{f(z')}{r'}\, dz',
\end{align}
and here again a \emph{harmonic annihilation} phenomenon should occur: the
leading problematic term disappears after integration by parts.

\subsection{Outline}

The outline of this article is as follows.  In Section~\ref{sec2}, we
briefly recall the estimate for the resolvent $(H+1)^{-1}$ built in
\cite{HL,GH1} and list several critical formulas.
Moreover, we verify the duality property on metric cones with
inverse-square potentials, namely that \eqref{RVp} implies
$(\textrm{RRV}_{p'})$; see Subsection~\ref{sec2.3}.

In Section~\ref{sec3}, we give a direct proof of
Theorem~\ref{thm1}, namely the Lorentz-type endpoint estimate for the
Riesz transform.  In Section~\ref{sec4}, we prove the reverse
inequality, i.e.\ Theorem~\ref{thm2}, by two different methods.  The
approach inspired by the observation \eqref{asy3} (the
\emph{exterior harmonic annihilation}) is presented in
Subsection~\ref{sec4.1}, while the \emph{interior harmonic
annihilation} method (inspired by \eqref{asy4}) and the endpoint estimates are developed in
Subsection~\ref{sec4.2}.  Finally, we analyze the formula obtained in
Subsection~\ref{sec4.2} and use it to prove the unboundedness of the
reverse inequality in Subsection~\ref{sec4.3}. Finally, in Appendix~\ref{sec5} we carry out the calculation of the asymptotic expansion mentioned in Subsection~\ref{sec1.3}.

\section{Preliminaries}\label{sec2}

\subsection{Resolvent on metric cone}\label{sec2.1}

Recall that on the metric cone $M = (0, \infty)_r \times Y$, the gradient operator can be written in the form
\begin{equation*}
    \nabla = \bigl( \partial_r,\, r^{-1} \nabla_Y \bigr).
\end{equation*}
Consider the quadratic form 
\begin{equation*}
    \mathcal{Q}f = \int_M \nabla f \cdot \nabla f \, d\mu,
\end{equation*}
where $\mu$ is the metric-induced measure $d\mu = r^{d-1} \, dr \, dh(y)$.

By the Friedrichs extension, there exists a unique non-negative self-adjoint operator $\Delta$ such that
\begin{equation*}
    \mathcal{Q}f = \int_M \Delta f \cdot f \, d\mu.
\end{equation*}
In particular, for $f\in C_c^\infty(M)$,
\begin{equation*}
    \Delta f(z) = \Delta f(r,y) = -\partial_r^2 f - \frac{d-1}{r}\,\partial_r f + \frac{\Delta_Y}{r^2} f,
\end{equation*}
where $\Delta_Y$ is the Laplace--Beltrami operator on $Y$.

Let $V_0$ be a smooth function on $Y$ satisfying \eqref{positivity}, and define the Schrödinger operator $H=\Delta+\frac{V_0}{r^2}$. We analyze the Riesz transform via the formula (with the inessential constant $2/\pi$ omitted)
\begin{equation*}
    \nabla H^{-1/2} = \int_0^\infty \nabla (H+\lambda^2)^{-1} \, d\lambda.
\end{equation*}
The crux, therefore, is to understand the resolvent operator $(H+\lambda^2)^{-1}$. Note that $H$ is homogeneous of degree $-2$. It follows that
\begin{equation*}
    (H+\lambda^2)^{-1}(z,z') = (H+\lambda^2)^{-1}(r,y; r',y') 
    = \lambda^{d-2} (H+1)^{-1}(\lambda r, y; \lambda r', y').
\end{equation*}

Let $\mu_j^2$ be the eigenvalues of $\Delta_Y+V_0(y)+\left(\frac{d-2}{2}\right)^2$, and let $u_j$ be the corresponding $L^2$-normalized eigenfunctions. By \cite[Section~4.3]{HL}, for $(z,z')$ sufficiently far away from the diagonal, the kernel $(H+1)^{-1}(\lambda r,y; \lambda r',y')$ has the explicit form
\begin{align}\label{formula1}
    \lambda^{2-d} (rr')^{1-d/2} \sum_{j\ge 0} u_j(y) \overline{u_j(y')} I_{\mu_j}(\lambda r) K_{\mu_j}(\lambda r'), \quad & 4r < r',\\ \label{formula2}
    \lambda^{2-d} (rr')^{1-d/2} \sum_{j\ge 0} u_j(y') \overline{u_j(y)} I_{\mu_j}(\lambda r') K_{\mu_j}(\lambda r), \quad & 4r' < r,
\end{align}
where $I,K$ are modified Bessel functions of the second type.

We emphasize that the original formula in \cite[Equation~4.11]{HL} is written in terms of half-densities, and it differs from the true resolvent kernel by the factor $(rr')^{1-d/2}$; see \cite[Equation~4.7, p.~495]{HL}. 

However, the above formulas do not converge near the diagonal. To remedy this, Hassell and Lin glue the formulas \eqref{formula1} and \eqref{formula2} together with a parametrix, to determine the behaviour near the diagonal. For details of the parametrix construction, we refer the reader to \cite[Sections~4.5--4.7]{HL}; see also \cite{GH1}.

Throughout the article, we fix a non-negative smooth function $\mathcal{X}:[0,\infty) \to [0,1]$ supported in $[0,8/9]$ such that $\mathcal{X}(t)=1$ for all $0\le t\le 1/2$. We define
\begin{align*}
    G_1(z,z') &= (H+1)^{-1}(z,z') \left( 1 - \mathcal{X}\left( \frac{4r}{r'} \right) - \mathcal{X}\left( \frac{4r'}{r} \right) \right),\\
    G_2(z,z') &= (H+1)^{-1}(z,z') \,\mathcal{X}\left( \frac{4r}{r'} \right),\\
    G_3(z,z') &= (H+1)^{-1}(z,z') \,\mathcal{X}\left( \frac{4r'}{r} \right).
\end{align*}
According to this decomposition, we split the Riesz transform $R$ into three parts, i.e.\ $R = R_1+R_2+R_3$, where 
\begin{equation}\label{Ri}
    R_i f(z) = \int_0^\infty \int_M \nabla_z G_i^\lambda(z,z') f(z') \,\lambda^{d-2} \, d\mu(z') \, d\lambda,
\end{equation}
and
\begin{equation*}
    G_i^\lambda(z,z') = G_i(\lambda z, \lambda z') = G_i(\lambda r, y; \lambda r', y').
\end{equation*}

We summarize the estimates from \cite{HL} in the following lemmata, which we will use throughout the paper.

\begin{lemma}\cite[Lemma~5.4, Prop~5.10 and~5.12]{HL}\label{leHL1}
We have the following kernel estimates:
\begin{enumerate}[label=\arabic*)]
    \item For $G_1$, 
    \begin{align*}
    \bigl|\nabla_{z,z'}^j G_1(\lambda r,y;\lambda r',y')\bigr| \lesssim 
    \begin{cases}
        \lambda^{2-d} \, d(z,z')^{2-d-j}, & d(z,z')\le \lambda^{-1},\\
        \lambda^{-N+j} \, d(z,z')^{-N}, & d(z,z')\ge \lambda^{-1},
    \end{cases}
    \end{align*}
    for all $N>0$ and $j\ge 0$.

    In addition, $R_1$ is of weak type $(1,1)$ and bounded on $L^p$ for all $1<p<\infty$.
    
    \item For $G_2$,
    \begin{align*}
        \bigl|G_2(\lambda r, y; \lambda r',y')\bigr| &\lesssim 
        \begin{cases}
            \lambda^{2-d} \, r^{1-d/2+\mu_0} r'^{1-d/2-\mu_0}, & \lambda \le r'^{-1},\\
            \lambda^{1-d/2+\mu_0-N} \, r^{1-d/2+\mu_0} r'^{-N}, & \lambda \ge r'^{-1},
        \end{cases} \\
        \bigl|\nabla_z G_2(\lambda r, y; \lambda r', y')\bigr| &\lesssim 
        \begin{cases}
            \lambda^{2-d} \, r^{-d/2+\mu_0} r'^{1-d/2-\mu_0}, & \lambda \le r'^{-1},\\
            \lambda^{1-d/2+\mu_0-N} \, r^{-d/2+\mu_0} r'^{-N}, & \lambda \ge r'^{-1},
        \end{cases}
    \end{align*}
    for all $N>0$.

    In addition, $R_2$ is $L^p$-bounded for all $p<p_1$.
    
    \item For $G_3$,
    \begin{align*}
        \bigl|G_3(\lambda r, y; \lambda r',y')\bigr| &\lesssim 
        \begin{cases}
            \lambda^{2-d} \, r^{1-d/2-\mu_0} r'^{1-d/2+\mu_0}, & \lambda \le r^{-1},\\
            \lambda^{1-d/2+\mu_0-N} \, r^{-N} r'^{1-d/2+\mu_0}, & \lambda \ge r^{-1},
        \end{cases} \\
        \bigl|\nabla_z G_3(\lambda r, y; \lambda r', y')\bigr| &\lesssim 
        \begin{cases}
            \lambda^{2-d} \, r^{-d/2-\mu_0} r'^{1-d/2+\mu_0}, & \lambda \le r^{-1},\\
            \lambda^{1-d/2+\mu_0-N} \, r^{-N-1} r'^{1-d/2+\mu_0}, & \lambda \ge r^{-1},
        \end{cases}
    \end{align*}
    for all $N>0$.

    In addition, $R_3$ is $L^p$-bounded for all $p>p_0$.
\end{enumerate}

\end{lemma}

\begin{remark}\label{remark1}
For later use, we give a similar estimate for $\nabla_{z'}G_2^\lambda$. In fact, by \cite[Prop~4.6, Theorem~4.11]{HL}, we know that $G_2$ is conormal at all boundary hypersurfaces. Hence, by conormality,
\begin{align}\label{dG2}
\bigl|\nabla_{z'} G_2(\lambda r, y; \lambda r', y')\bigr| &\lesssim 
        \begin{cases}
            \lambda^{2-d} \, r^{1-d/2+\mu_0} r'^{-d/2-\mu_0}, & \lambda \le r'^{-1},\\
            \lambda^{1-d/2+\mu_0-N} \, r^{1-d/2+\mu_0} r'^{-N-1}, & \lambda \ge r'^{-1},
        \end{cases}
\end{align}
for all $N>0$.
\end{remark}

\begin{lemma}\cite[Lemma~4.4]{HL}\label{leHL2}
For any $0<\beta<1$ and any $M,N>0$, the series
\begin{equation*}
    \sum_{\mu_j \ge M} \mu_j^N \alpha^{\mu_j-M}
\end{equation*}
converges for all $0<\alpha \le \beta$, and it is bounded uniformly in $\alpha$.
\end{lemma}

The following lemma was also studied in \cite{GH1}.

\begin{lemma}\cite[Corollary~5.9]{HL}\label{leHL3}
Let $K$ be a kernel defined on $M=(0,\infty)_r \times Y$. Suppose 
\begin{equation*}
    |K(r,y;r',y')| \lesssim 
    \begin{cases}
        r^{-\alpha} r'^{-\beta}, & r\le r',\\
        0, & r>r',
    \end{cases}
\end{equation*}
and assume that $\alpha+\beta=d$, $\beta>0$, and $p<\dfrac{d}{\max(\alpha,0)}$. Then $K$ acts as a bounded operator on $L^p(M, r^{d-1}dr\,dh)$.

Similarly, if 
\begin{equation*}
    |K(r,y;r',y')| \lesssim 
    \begin{cases}
        0, & r\le r',\\
        r^{-\gamma} r'^{-\delta}, & r>r',
    \end{cases}
\end{equation*}
and $\gamma+\delta=d$, $\gamma>0$ and $p>\dfrac{d}{\min(\gamma,d)}$, then $K$ acts as a bounded operator on $L^p(M, r^{d-1}dr\,dh)$.
\end{lemma}

\subsection{Hardy’s inequality}

\begin{lemma}\label{Hardy}
On $M$, the following Hardy inequality holds:
\begin{align*}
    \int_{M} \left| \frac{f(r,y)}{r} \right|^p \, d\mu 
    \le C \int_M |\nabla f(r,y)|^p \, d\mu,\quad \forall f\in C_c^\infty(M),
\end{align*}
for all $p\in (1,\infty)\setminus \{d\}$. 
\end{lemma}

\begin{proof}[Proof of Lemma~\ref{Hardy}]
Since $Y$ is compact, it suffices to prove that for each $y\in Y$,
\begin{align*}
    \int_0^\infty \left| \frac{f(r,y)}{r} \right|^p r^{d-1}\,dr 
    \lesssim \int_0^\infty |\partial_r f(r,y)|^p r^{d-1} \, dr,
\end{align*}
with an implicit constant independent of $y$. The one-dimensional Hardy's inequality with weight (see, for example, \cite{Grafakos1}) gives precisely this estimate, provided $p\neq d$.
\end{proof}

\begin{remark}
    The exclusion of the case $p=d$ is sharp in the weighted one-dimensional Hardy's inequality with weight: at $p=d$ the inequality fails in general, which explains why $p=d$ is omitted from Lemma~\ref{Hardy} and Corollary~\ref{cor1}.
\end{remark}

It follows that:

\begin{proof}[Proof of Corollary~\ref{cor1}]
The result follows directly by combining Theorem~\ref{thm2} with Lemma~\ref{Hardy}.
\end{proof}

\subsection{Duality property between Riesz and reverse Riesz inequalities}\label{sec2.3}

In this subsection, we identify the most natural conjectural range for \eqref{RRVp} to hold, namely the dual range associated with the $L^p$-boundedness of the Riesz transform. 

To proceed, we define
\begin{align*}
    T u := r^{-1} H^{-1/2}u 
    = r^{-1} \int_0^\infty (H+\lambda^2)^{-1} u \, d\lambda.
\end{align*}
Using the parametrix construction, we decompose $T$ into three pieces, $T = \sum_{j=1}^3 T_j$, where
\begin{align*}
    T_j u(z) 
    := r^{-1} \int_0^\infty \int_M \lambda^{d-2} G_j(\lambda z, \lambda z') \, u(z') \, d\mu(z') \, d\lambda,
\end{align*}
for $j=1,2,3$.

\begin{lemma}\label{T1}
$T_1$ is bounded on $L^p$ for all $1 < p < \infty$.
\end{lemma}

\begin{proof}[Proof of Lemma~\ref{T1}]
By Lemma~\ref{leHL1}, the on-diagonal part $G_1$ has estimate
\begin{equation*}
    |G_1(\lambda z, \lambda z')|\lesssim \begin{cases}
        \lambda^{2-d} d(z,z')^{2-d}, & \lambda \le d(z,z')^{-1},\\
        \lambda^{-N} d(z,z')^{-N}, & \lambda \ge d(z,z')^{-1},
    \end{cases}
\end{equation*}
for all $N>0$.

Therefore, by choosing $N>d-1$, the kernel of $T_1$ can be estimated by
\begin{align*}
    \left|r^{-1} \int_0^\infty \lambda^{d-2} G_1(\lambda z, \lambda z') d\lambda \right| &\lesssim r^{-1} \int_0^{d(z,z')^{-1}} d(z,z')^{2-d} d\lambda\\
    &+ r^{-1} \int_{d(z,z')^{-1}}^\infty \lambda^{d-2-N} d(z,z')^{-N} d\lambda  \\
    &\lesssim r^{-1} d(z,z')^{1-d}.
\end{align*}
Combining this with the fact that $G_1$ is supported in the range $\{r/8 \le r' \le 8r\}$, we conclude 
\begin{equation*}
    |T_1u(z)|\lesssim r^{-1} \int_{\{r/8 \le r' \le 8r\}} \frac{|u(z')|}{d(z,z')^{d-1}} dz'.
\end{equation*}
Next, we replace the domain of integration by $B(z,10r)$. Indeed, for $z'=(r',y')$ such that $r/8 \le r' \le 8r$, we have
\begin{equation*}
    d(z,z')^2 = \begin{cases}
        |r-r'|^2 + 2rr' \left(1-\cos{d_Y(y,y')}\right), & d_Y(y,y')\le \pi,\\
        (r+r')^2, & d_Y(y,y')\ge \pi,
    \end{cases} < 100 r^2.
\end{equation*}
To this end, we proceed with the following standard argument.
\begin{align*}
    |T_1u(z)|&\lesssim r^{-1} \int_{B(z,10r)} \frac{|u(z')|}{d(z,z')^{d-1}} dz' \lesssim r^{-1} \sum_{i=0}^\infty \int_{10r/2^{i+1}\le d(z,z')<10r/2^i} \left(\frac{r}{2^i}\right)^{1-d} |u(z')| dz'\\
    &\lesssim r^{-d} \sum_{i\ge 0} 2^{i(d-1)} \int_{B(z,10r/2^i)}|u(z')| dz' \lesssim r^{-d} \sum_{i\ge 0} 2^{i(d-1)} \left(\frac{r}{2^i}\right)^d \mathcal{M}u(z)\lesssim \mathcal{M}u(z),
\end{align*}
where $\mathcal{M}u(z) = \sup_{B \ni z} \textrm{vol}(B)^{-1} \int_B |u|$ is the usual Hardy-Littlewood maximal operator and the second-to-last inequality is due to the fact that $\textrm{vol}(B(z,r)) \sim r^d$ for all $z\in M$ and $r>0$. The proof follows directly by the maximal theorem.

\end{proof}

\begin{lemma}\label{T2T3}
$T_2$ is bounded on $L^p$ for $1<p<p_1$ and $T_3$ is bounded on $L^p$ for $p_0<p<\infty$, where $p_0,p_1$ are defined as in \eqref{p0p1}.
\end{lemma}

\begin{proof}[Proof of Lemma~\ref{T2T3}]
By Lemma~\ref{leHL1}, $G_2$ has upper bound:
\begin{equation*}
|G_2(\lambda r, \lambda r',y,y')|  \lesssim \begin{cases}
    \lambda^{2-d} r^{1-d/2+\mu_0} r'^{1-d/2-\mu_0}, & \lambda \le r'^{-1},\\
    \lambda^{1-d/2+\mu_0-N} r^{1-d/2+\mu_0} r'^{-N}, & \lambda \ge r'^{-1},
\end{cases}
\end{equation*}
for all $N>0$. It follows that
\begin{align*}
    |T_2(z,z')| \lesssim &r^{-1} \int_0^{r'^{-1}} r^{1-d/2+\mu_0} r'^{1-d/2-\mu_0} d\lambda\\
    &+ r^{-1} \int_{r'^{-1}}^\infty  \lambda^{d/2-1+\mu_0-N} r^{1-d/2+\mu_0} r'^{-N} d\lambda\\
    &\lesssim r^{-d/2+\mu_0} r'^{-d/2-\mu_0},
\end{align*}
where we choose $N>d/2+\mu_0$.

Similarly, a parallel argument yields 
\begin{align*}
    |T_3(z,z')|\lesssim &r^{-1} \int_0^{r^{-1}} r^{1-d/2-\mu_0} r'^{1-d/2+\mu_0} d\lambda \\
    &+ r^{-1} \int_{r^{-1}}^\infty \lambda^{d/2-1+\mu_0-N} r^{-N} r'^{1-d/2+\mu_0} d\lambda\\
    &\lesssim r^{-1-d/2-\mu_0} r'^{1-d/2+\mu_0}.
\end{align*}
Recall that $G_2$ is supported in the range where $r\le r'$ and $G_3$ is supported in $\{r\ge r'\}$. $T_2u$ and $T_3u$ are therefore bounded by
\begin{align*}
    r^{-d/2+\mu_0} \int_{\{r\le r'\}} r'^{-d/2-\mu_0} |u(z')| dz' \quad \textrm{and} \quad r^{-1-d/2-\mu_0} \int_{\{r\ge r'\}} r'^{1-d/2+\mu_0} |u(z')| dz',
\end{align*}
respectively. The proof follows from Lemma~\ref{leHL3} immediately.

\end{proof}

\begin{remark}
It is not surprising that $T_2$ and $T_3$ have the same range of $L^p$-boundedness as $R_2$ and $R_3$ in the Riesz transform. Indeed, since $G_2$ and $G_3$ are polyhomogeneous conormal at all boundary hypersurfaces, there is essentially no difference, at the level of $L^p$-mapping properties, between applying $\nabla_z$ to the kernel and multiplying it by $r^{-1}$.
\end{remark}

Next, we verify the duality property between \eqref{RVp} and $(\textrm{RRV}_{p'})$.

\begin{proposition}\label{lemma2}
Under the assumptions of Theorem~\ref{thm2}, the reverse inequality holds for $p\in (p_1', p_0')$. That is 
\begin{align*}
    \|H^{1/2}f\|_p \le C \|\nabla_H f\|_p, \quad p_1'<p<p_0',
\end{align*}
for all $f\in C_c^\infty(M)$.
\end{proposition}

\begin{proof}[Proof of Proposition~\ref{lemma2}]
Let $g\in C_c^\infty(M)$ with $\|g\|_{p'}\le 1$. By duality, it is enough to consider bilinear form $\langle H^{1/2}f, g\rangle$. Note that
\begin{align*}
    \langle H^{1/2} f, g \rangle &= \left\langle \int_0^\infty (H+\lambda^2)^{-1} Hf d\lambda, g \right\rangle = \left\langle Hf, \int_0^\infty (H+\lambda^2)^{-1} g d\lambda \right\rangle\\
    &= \left\langle \nabla f, \nabla \int_0^\infty (H+\lambda^2)^{-1} g d\lambda \right\rangle + \left\langle \frac{V_0}{r} f, r^{-1}\int_0^\infty (H+\lambda^2)^{-1} g d\lambda \right\rangle\\
    &\le \|\nabla f\|_p \|Rg\|_{p'} + C \left\|\frac{f}{r} \right\|_p \sum_{j=1}^3 \|T_j g\|_{p'}\\
    &\le C \|\nabla_H f\|_p \|g\|_{p'},
\end{align*}
where the last inequality follows by Theorem~\ref{thmHL}, Lemma~\ref{T1} and Lemma~\ref{T2T3}. The proof is complete. 

\end{proof}

To end this section, we give a proof for Corollary~\ref{cor2}. 

\begin{proof}[Proof of Corollary~\ref{cor2}]
By \eqref{d_H} and Theorem~\ref{thmHL}, it suffices to consider the $L^p$-boundedness of $T$. The result follows immediately from Lemma~\ref{T1} and Lemma~\ref{T2T3}.
\end{proof}

\section{Endpoint estimates: Proof of Theorem~\ref{thm1}}\label{sec3}

In this section we prove Theorem~\ref{thm1}; namely, we show that the Riesz transform $R$ is of \emph{restricted weak type} at both endpoint exponents $p_0$ and $p_1$. 
Throughout, “restricted weak type $(p,p)$’’ means that for all measurable sets $E\subset M$ and all $\alpha>0$,
\begin{equation*}
\mu\left( \left\{z: |R \mathbf{1}_{E}(z)| > \alpha \right\} \right) \le C \alpha^{-p} \mu(E),
\end{equation*}
or equivalently,
\begin{equation*}
    \| R f \|_{(p_i,\infty)} \lesssim \| f \|_{(p_i,1)},\quad \textrm{for}\quad i=0,1.
\end{equation*}

We use the parametrix decomposition from Section~\ref{sec2}: write
\[
R=\int_0^\infty \nabla (H+\lambda^2)^{-1}\,d\lambda
= R_1+R_2+R_3,
\]
where $R_i$ corresponds to $G_i$ and inherits the kernel bounds in Lemma~\ref{leHL1}. 
The near-diagonal term $R_1$ is a Calderón–Zygmund operator: the size and smoothness estimates for $\nabla G_1$ yield weak type $(1,1)$ and $L^2$-boundedness; hence $R_1$ is harmless for our endpoint analysis.

The off-diagonal terms $R_2$ and $R_3$ capture the conic asymptotics and determine the endpoints. Using the factorized kernel bounds for $G_2$ and $G_3$ in Lemma~\ref{leHL1} together with the Hardy-Littlewood-P\'olya type mapping criteria in Lemma~\ref{leHL3}, we obtain restricted weak type $(p_1,p_1)$ and $(p_0,p_0)$, respectively. 

Combining these estimates for $R_2$ and $R_3$ with the Calderón–Zygmund bounds for $R_1$ completes the proof of Theorem~\ref{thm1}.

\subsection{Lorentz-type endpoint estimate}

\begin{lemma}\label{leR}
$R$ is of restricted weak type $(p_0,p_0)$ and $(p_1,p_1)$.
\end{lemma}

\begin{proof}[Proof of Lemma~\ref{leR}]
Our assumption $p_1<\infty$ implies $0<\mu_0<\tfrac{d}{2}$. By spectral theory and the decomposition discussed above,
\[
\nabla H^{-1/2} \;=\; \int_0^\infty \nabla(H+\lambda^2)^{-1}\,d\lambda \;=\; R_1+R_2+R_3,
\]
where $R_i$ are defined as in \eqref{Ri}.

By Lemma~\ref{leHL1}, $R_1$ satisfies Calderón–Zygmund estimates; hence it is of weak type $(1,1)$ and bounded on $L^p$ for all $p\in(1,\infty)$. In particular, for the endpoints we will need,
\[
\|R_1\|_{L^{p_i,1}\to L^{p_i,\infty}} \le C, \qquad i=0,1.
\]
(The special case $p_0=1$ is included, since weak type $(1,1)$ coincides with restricted weak type $(1,1)$). Next, by Lemma~\ref{leHL3} (together with Lemma~\ref{leHL1}) we already have strong $L^{p}$-boundedness of $R_3$ at $p=p_1$ and of $R_2$ at $p=p_0$. It therefore remains to establish the complementary endpoint estimates:
\begin{itemize}
    \item if $p_0>1$, show that $R_3$ is of restricted weak type $(p_0,p_0)$ and $R_2$ is of restricted weak type $(p_1,p_1)$ if $p_1<\infty$.
    \item if $p_0=1$, show that $R_2, R_3$ are both of weak type $(1,1)$ (while $R_2$ remains of restricted weak type $(p_1,p_1)$ if $p_1<\infty$).
\end{itemize}
These two bounds, together with the estimate for $R_1$, yield Theorem~\ref{thm1}.

We first consider the endpoint $p_0$. By Lemma~\ref{leHL1} and Hölder's inequality,
\begin{align*}
    |R_3(f)(r,y)| &\lesssim \int_{Y\times (0,\infty)} r^{-\mu_0-d/2-1} r'^{\mu_0-d/2+1} \mathbf{1}_{r'\le r}(r,r') f(r',y') r'^{d-1} dr' dh(y')\\
    &\lesssim r^{-\mu_0-d/2-1} \|f\|_{(p_0,1)} \|H_r\|_{(p_0',\infty)},
\end{align*}
where $H_r(r',y') = r'^{\mu_0-d/2+1} \mathbf{1}_{r'\le r}(r,r')$. 

Suppose first $0<\mu_0<d/2-1$, i.e., $p_0=\frac{d}{1+d/2+\mu_0}$. A straightforward calculation yields 
\begin{align*}
\|H_r\|_{(p_0',\infty
)}^{p_0'} &= \sup_{\lambda>0} \lambda^{p_0'} \mu\left(\{M; (r')^{\mu_0-d/2+1} \mathbf{1}_{r'\le r}(r,r')>\lambda\}\right)\\
&\lesssim \sup_{\lambda>0} \begin{cases}
    \lambda^{\frac{d}{d/2-1-\mu_0}} \lambda^{\frac{d}{\mu_0-d/2+1}}, & \lambda \ge r^{\mu_0-d/2+1}\\
    \lambda^{\frac{d}{d/2-1-\mu_0}} r^d, & \lambda < r^{\mu_0-d/2+1}
\end{cases}\\
&\lesssim 1,
\end{align*}
since $Y$ has finite measure and $p_0' = \frac{d}{d/2-1-\mu_0}$. It follows that 
\begin{equation*}
 |R_3(f)(r,y)|\lesssim r^{-\mu_0-d/2-1} \|f\|_{(p_0,1)}.
\end{equation*}
One checks easily that $r^{-\mu_0-d/2-1}\in L^{p_0,\infty}(M)$, which shows the restricted $(p_0,p_0)$ boundedness of $R_3$. Next, for $\mu_0 \ge d/2-1$ (i.e., $p_0=1$), we have to establish that $R_2, R_3$ are both of weak type $(1,1)$. Note that
\begin{align*}
    |R_2(f)(r,y)|&\lesssim \int_{Y\times (0,\infty)} r^{\mu_0-d/2} r'^{-\mu_0-d/2}\mathbf{1}_{r\le r'}(r,r') f(r',y') r'^{d-1}dr' dh(y')\\
    &\lesssim r^{\mu_0-d/2} \|f\|_{(p_1,1)} \|F_r\|_{(p_1',\infty)},
\end{align*}
where $F_r(r',y') = r'^{-\mu_0-d/2}\mathbf{1}_{r\le r'}(r,r')$.

It is then plain that
\begin{align*}
    |(R_2+R_3)(r,y)|&\lesssim \left(r^{\mu_0-d/2} \|F_r\|_\infty + r^{-\mu_0-d/2-1} \|H_r\|_\infty \right) \|f\|_1\\
    &\lesssim \left(r^{\mu_0-d/2} r^{-\mu_0-d/2} + r^{-\mu_0-d/2-1} r^{\mu_0-d/2+1} \right) \|f\|_1\\
    &\lesssim r^{-d} \|f\|_1,
\end{align*}
and the weak type $(1,1)$ statement is verified since $r^{-d}\in L^{1,\infty}(M)$.

To this end, we consider the case $p_1<\infty$ and claim that $R_2$ is restricted weak $(p_1,p_1)$. Directly, we have
\begin{align*}
    \|F_r\|_{(p_1',\infty)}^{p_1'} &= \sup_{\lambda>0} \lambda^{p_1'} \mu\left(\{M; r'^{-\mu_0-d/2}\mathbf{1}_{r\le r'}(r,r')>\lambda\}\right)\\
    &\lesssim \sup_{0<\lambda<r^{-\mu_0-d/2}} \lambda^{p_1'}(\lambda^{-\frac{d}{\mu_0+d/2}} - r^{d})\\
    &\lesssim 1,
\end{align*}
since $p_1'=\frac{d}{\mu_0+d/2}$. Thus $|R_2(f)(r,y)|\lesssim r^{\mu_0-d/2} \|f\|_{(p_1,1)}$ and the claim follows since $r^{\mu_0-d/2}\in L^{p_1,\infty}(M)$. The proof of Lemma~\ref{leR} is now complete.

\end{proof}

\begin{remark}\label{remark_cone}
We note that the above argument also applies to the special case $V=0$ (i.e.\ the case considered in \cite{HQLi}). As discussed in \cite[Section~5.5]{HL}, the contribution of the terms corresponding to $\mu_0$ (which in this case equals $\frac{d-2}{2}$) to the kernel of the Riesz transform defines a bounded operator on $L^p$ for all $p\in (1,\infty)$. Hence the proof of Lemma~\ref{leR} can be applied verbatim after replacing $\mu_0$ by $\mu_1$ (the square root of the second smallest eigenvalue). 

The corresponding statement is that the Riesz transform $\nabla \Delta^{-1/2}$ on the metric cone $M = (0,\infty)\times Y$ is of restricted weak type $(p_1,p_1)$, where
\begin{equation*}
    p_1 = \frac{d}{\max \left(\frac{d}{2}-\mu_1, 0\right)}.
\end{equation*}
\end{remark}

\subsection{Counterexamples for sharpness}

To complete the proof of Theorem~\ref{thm1}, we also need the negative (sharpness) statement. 
Following the idea of \cite{GH1}, we prove:

\begin{lemma}\label{lenegative}
For $i\in\{0,1\}$ and each $q>1$, there exists a function 
$f\in L^{p_i,q}(M)$ such that $|Rf(z)|=+\infty$ for almost every $z\in M$.
\end{lemma}

In particular, at the endpoints $p=p_i$ the restricted weak type bound of Theorem~\ref{thm1} is optimal in the Lorentz scale: $R$ is bounded on $L^{p_i,1}\!\to L^{p_i,\infty}$, but this fails for every larger Lorentz space $L^{p_i,q}$ with $q>1$.

\begin{proof}[Proof of Lemma~\ref{lenegative}]

By formulas \eqref{formula1} and \eqref{formula2}, the problematic terms (i.e. $R_2$, $R_3$) have expansion

\begin{align}\label{formula3}
    R_2f(z) = \nabla_z \int_M \int_0^\infty (rr')^{1-d/2} \sum_{j\ge 0} u_j(y) \overline{u_j(y')} I_{\mu_j}(\lambda r) K_{\mu_j}(\lambda r') \mathcal{X}\left( \frac{4r}{r'} \right)d\lambda f(z') dz', \\ \label{formula4}
    R_3f(z) = \nabla_z \int_M \int_0^\infty  (rr')^{1-d/2} \sum_{j\ge 0} u_j(y') \overline{u_j(y)} I_{\mu_j}(\lambda r') K_{\mu_j}(\lambda r) \mathcal{X}\left( \frac{4r'}{r} \right) d\lambda f(z') dz'.
\end{align}
By the $L^p$-boundedness of $R_2, R_3$, it suffices to check the unboundedness of $R_2$ near $p=p_1$ and $R_3$ near $p=p_0$. 

First, since $\nabla = (\partial_r, r^{-1} \nabla_Y)$, it suffices to consider the tangential derivative $r^{-1} \nabla_Y$. The corresponding entries of \eqref{formula3} and \eqref{formula4} for each $j$ are
\begin{align}\label{formula5}
     r^{-d/2} \nabla_Y u_j(y) \int_M \left(\int_0^\infty  I_{\mu_j}(\lambda r) K_{\mu_j}(\lambda r') d\lambda \right)  r'^{1-d/2} \overline{u_j(y')} \mathcal{X}\left( \frac{4r}{r'} \right) f(z') dz', 
\end{align}
and
\begin{align}\label{formula6}
     r^{-d/2} \nabla_Y\overline{u_j(y)} \int_M  \left(\int_0^\infty  I_{\mu_j}(\lambda r') K_{\mu_j}(\lambda r) d\lambda \right) r'^{1-d/2}  u_j(y') \mathcal{X}\left( \frac{4r'}{r} \right) f(z') dz'.
\end{align}

Second, we show that after extracting finitely many terms, the remainders of \eqref{formula5} and \eqref{formula6} act as bounded operators on $L^p$ for all $1<p<\infty$. Indeed, by Lemma~\ref{leHL2} and Hörmander's $L^\infty$-estimate \cite{Hormander1}, for $J\ge 0$,
\begin{align*}
    \sum_{j=J}^\infty |\nabla_Y u_j(y)| |I_{\mu_j}(\lambda r)| |K_{\mu_j} (\lambda r')| |u_j(y')| \mathbf{1}_{r\le \frac{2}{9}r'} \lesssim \sum_{j\ge J} \mu_j^d \left(\frac{4r}{r'}\right)^{\mu_j} e^{-\lambda r'/4} \lesssim r^{\mu_J} r'^{-\mu_J} e^{-c\lambda r'}.
\end{align*}
Similarly,
\begin{align*}
    \sum_{j=J}^\infty |\nabla_Y u_j(y)| |I_{\mu_j}(\lambda r')| |K_{\mu_j} (\lambda r)| |u_j(y')| \mathbf{1}_{r'\le \frac{2}{9}r} \lesssim \sum_{j\ge J} \mu_j^d \left(\frac{4r'}{r}\right)^{\mu_j} e^{-\lambda r/4} \lesssim r^{-\mu_J} r'^{\mu_J} e^{-c\lambda r}.
\end{align*}
It follows that
\begin{align}\label{formula7}
    \left| \sum_{j\ge J} \eqref{formula5} \right| \lesssim r^{\mu_J - d/2} \int_{r\le r'} r'^{-\mu_J - d/2} |f(z')| dz', 
\end{align}
and
\begin{align}\label{formula8}
    \left| \sum_{j\ge J} \eqref{formula6} \right| \lesssim r^{-1-d/2-\mu_J} \int_{r'\le r} r'^{1-d/2+\mu_J} |f(z')| dz'.
\end{align}
Now, by Lemma~\ref{leHL3}, it is clear that $\eqref{formula7}$ acts as a bounded operator on $L^p$ for $p< \frac{d}{\max \left( \frac{d}{2}-\mu_J, 0 \right)}$. Meanwhile, $\eqref{formula8}$ is $L^p$-bounded for $p>\frac{d}{\min\left( \frac{d+2}{2}+\mu_J, d \right)}$. It follows by Weyl's law that there exists $J\ge 0$ such that both \eqref{formula7} and \eqref{formula8} are $L^p$-bounded for all $1<p<\infty$. Hence, it is enough to treat the first $J$ terms of the sum.

Let $q>1$. Consider function $f(z') = u_0(y') r'^{\mu_0-d/2} \left( 1 + |\log{r'}| \right)^{-s} $ with $1/q < s\le 1$. We show that $R_2f$ diverges a.e. but $\|f\|_{(p_1,q)}<\infty$. Observe that by \cite[Page~684, 6.576.5]{GR},
\begin{align*}
    \int_0^\infty  I_{\mu_j}(\lambda r) K_{\mu_j}(\lambda r') d\lambda &= C_{\mu_j} r'^{-1} \left(\frac{r}{r'}\right)^{\mu_j}  {}_2 F_1 \left( \frac{1}{2}+\mu_j, \frac{1}{2}; \mu_j+1; \left(\frac{r}{r'}\right)^2 \right) \quad &&r<r', \\
    \int_0^\infty  I_{\mu_j}(\lambda r') K_{\mu_j}(\lambda r) d\lambda &= C_{\mu_j} r^{-1} \left(\frac{r'}{r}\right)^{\mu_j} {}_2 F_1 \left( \frac{1}{2}+\mu_j, \frac{1}{2}; \mu_j+1; \left(\frac{r'}{r}\right)^2 \right)\quad &&r>r',
\end{align*}
where $C_{\mu_j}>0$ is uniformly bounded depending only on $\mu_j$. Moreover, since $0<r/r'<1/4$ and ${}_2F_1(a,b;c;t)$ is increasing if $a,b,c>0$, we have $1={}_2F_1(a,b;c;0) \le {}_2F_1(a,b;c;(r/r')^2) \le {}_2F_1(a,b;c;1/16) = C$. Hence, by the orthogonality and normalization of $u_j$, we have for a.e. $(r,y)\in M$,
\begin{align*}
    \left|\sum_{j=0}^J \eqref{formula5} \right| \ge C r^{\mu_0 - d/2} |\nabla_Y u_0(y)| \int_0^\infty \mathcal{X}\left(\frac{4r}{r'}\right) \left( 1 + |\log{r'}| \right)^{-s} \frac{dr'}{r'},
\end{align*}
and the integral diverges since 
\begin{equation*}
    \int_0^\infty \mathcal{X}\left(\frac{4r}{r'}\right) \left( 1 + |\log{r'}| \right)^{-s} \frac{dr'}{r'} \ge \int_{8r}^\infty (1+|\log{t}|)^{-s} \frac{dt}{t} \ge \int_{|\log{8r}|}^\infty (1+x)^{-s} dx = \infty, \quad \forall s\le 1.
\end{equation*}
Meanwhile, one checks easily that for any $q>1$,
\begin{align*}
    \|f\|_{(p_1,q)}^p \sim \int_0^\infty \left( t^{1/p_1} f^*(t) \right)^q \frac{dt}{t} \lesssim \int_0^\infty t^{\frac{q}{p_1}} t^{\frac{q(2\mu_0-d)}{2d}} (1+|\log{t}|)^{-sq} \frac{dt}{t}\\
    = \int_0^\infty (1+x)^{-sq} dx < \infty
\end{align*}
since $s>1/q$. This confirms the unboundedness of $R_2$ acting on $L^{p_1,q}$ for any $q>1$.

The argument for the unboundedness of $R_3$ acting on $L^{p_0,q}$ is similar. Indeed, choose $f(z') = \overline{u_0(y')} r'^{-\frac{d+2}{2}-\mu_0} (1+|\log{r'}|)^{-s}$ with $1/q<s\le 1$. Then
\begin{align*}
    \left|\sum_{j=0}^J \eqref{formula6}\right| &\ge  r^{-1-d/2-\mu_0} |\nabla_Y \overline{u_0(y)}| \int_0^\infty \mathcal{X}\left(\frac{4r'}{r}\right) (1+|\log{r'}|)^{-s} \frac{dr'}{r'}\\
    &\ge r^{-1-d/2-\mu_0} |\nabla_Y \overline{u_0(y)}| \int_{|\log{8r'}|}^\infty (1+x)^{-s} dx = \infty
\end{align*}
for a.e. $(r,y)\in M$. On the other hand,
\begin{equation*}
    \|f\|_{(p_0,q)}^q \sim \int_0^\infty \left( t^{1/p_0} f^*(t) \right)^q \frac{dt}{t} \lesssim \int_0^\infty t^{\frac{q}{p_0}} t^{-\frac{q(2\mu_0+d+2)}{2d}} (1+|\log{t}|)^{-sq} \frac{dt}{t}
    = \int_0^\infty (1+x)^{-sq} dx < \infty
\end{equation*}
since again $s>1/q$, completing the proof.
\end{proof}

\section{Reverse Riesz problem: Proof of Theorem~\ref{thm2}}\label{sec4}

In this section, we prove our main result, Theorem~\ref{thm2} in two different ways with the same core idea, according to \eqref{asy3} and \eqref{asy4}, respectively. Note that the crux is to perform the \emph{harmonic annihilation} (integration by parts) at the right situation, which in our case, it should take place on $G_2$ part in the bilinear form from outside or on $G_3$ part from inside; see Subsection~\ref{sec1.3}.

\subsection{Exterior harmonic annihilation}\label{sec4.1}

To to so, we prove:

% Set $k_\alpha(r):= r^{1-/d/2} K_{|\alpha|}(r) $ and $l_\alpha(r) := r^{1-d/2} I_\alpha(r)$.

\begin{lemma}\label{HG}
We have kernel estimates:
\begin{equation}\label{HG1}
    |H_z G_2^\lambda(z,z')| \lesssim \lambda^{2-d} r^{-d} e^{-c\lambda r'} \mathbf{1}_{r'/8 \le r \le r'/4} + \mathbf{1}_{r\le r'/4}\begin{cases}
        \lambda^{4-d} r^{1-d/2+\mu_0} r'^{1-d/2-\mu_0}, & \lambda \le r'^{-1},\\
        \lambda^{3-d/2+\mu_0-N} r^{1-d/2+\mu_0} r'^{-N}, & \lambda \ge r'^{-1},
    \end{cases}
\end{equation}
and by symmetry,
\begin{equation}\label{HG2}
 |H_{z'} G_3^\lambda(z,z')| \lesssim \lambda^{2-d} r'^{-d} e^{-c\lambda r} \mathbf{1}_{r/8 \le r' \le r/4} + \mathbf{1}_{r'\le r/4}\begin{cases}
        \lambda^{4-d} r'^{1-d/2+\mu_0} r^{1-d/2-\mu_0}, & \lambda \le r^{-1},\\
        \lambda^{3-d/2+\mu_0-N} r'^{1-d/2+\mu_0} r^{-N}, & \lambda \ge r^{-1},
    \end{cases}
\end{equation}
for all $N>0$.

\end{lemma}

\begin{proof}[Proof of Lemma~\ref{HG}]
By symmetry, it is enough to prove \eqref{HG1}, and then \eqref{HG2} follows directly by switching the roles of $r$ and $r'$.

For $G_2^\lambda(z,z')$, we have explicit formula:
\begin{equation}\label{HG_G2}
    \lambda^{2-d} (rr')^{1-d/2} \sum_{j\ge 0} u_j(y) \overline{u_j(y')} I_{\mu_j}(\lambda r) K_{\mu_j}(\lambda r') \mathcal{X}\left(\frac{4r}{r'}\right).
\end{equation}
Introduce notation $\mathcal{I}_\alpha(x) = x^{1-\alpha/2} I_{\alpha/2-1}(x)$. It is known that $\mathcal{I}_\alpha(\lambda r)$ solves equation
\begin{equation*}
    f''(r) + \frac{\alpha-1}{r} f'(r) = \lambda^2 f(r);
\end{equation*}
see for example \cite[Section~3.1]{HS1D} or \cite[Section~9.6.1]{AS}. Denote by $\Delta_\alpha$ the weighted 'Laplacian':
\begin{equation*}
    \Delta_\alpha = - \partial_r^2 - \frac{\alpha-1}{r} \partial_r.
\end{equation*}
Hence, $\Delta_\alpha \mathcal{I}_\alpha(\lambda r) = - \lambda^2 \mathcal{I}_\alpha(\lambda r)$. Below, to lighten notation we suppress the index and write $\mu$ and $u$ for $\mu_j$ and $u_j$, respectively.

Write each term in the sum \eqref{HG_G2} as
\begin{align*}
    \lambda^{2-d+\mu} r^{\mu+1-d/2} r'^{1-d/2} \mathcal{I}_{2\mu+2}(\lambda r) K_{\mu}(\lambda r') u(y) \overline{u(y')} \mathcal{X}\left(\frac{4r}{r'}\right)
\end{align*}
and then apply $H_z = \Delta_z + \frac{V_0(y)}{r^2} = -\partial_r^2 - \frac{d-1}{r} \partial_r + \frac{\Delta_Y + V_0}{r^2}$ to it, obtaining 
\begin{align}\label{HG_G2_expand}
 \lambda^{2-d+\mu} r'^{1-d/2} \Delta_d\left( \mathcal{I}_{2\mu+2}(\lambda r) r^{\mu+1-d/2}\mathcal{X}\left(\frac{4r}{r'}\right) \right) K_{\mu}(\lambda r') u(y) \overline{u(y')} \\ \nonumber
+ \lambda^{2-d} (rr')^{1-d/2} \frac{\left(\Delta_Y+V_0(y) \right) u(y)}{r^2} I_\mu(\lambda r) K_\mu(\lambda r') \overline{u(y')} \mathcal{X}\left(\frac{4r}{r'}\right).
\end{align}
Since $\left(\Delta_Y+V_0(y) + \left(\frac{d-2}{2}\right)^2 \right) u(y) = \mu^2 u(y)$, the second term above equals
\begin{equation}\label{HG_G2_expand2}
\lambda^{2-d} (rr')^{1-d/2} \frac{\mu^2 - \left(\frac{d-2}{2}\right)^2}{r^2} I_\mu(\lambda r) K_\mu(\lambda r') u(y) \overline{u(y')} \mathcal{X}\left(\frac{4r}{r'}\right).
\end{equation}
Next, we treat the first term of \eqref{HG_G2_expand}. By splitting $\Delta_d = \Delta_{2\mu+2} + \frac{2\mu+2-d}{r}\partial_r$, the $\Delta_d (\dots)$ part of the first term of \eqref{HG_G2_expand} can be expanded as
\begin{align*}
    \Delta_{2\mu+2} \mathcal{I}_{2\mu+2}(\lambda r) r^{\mu+1-d/2} \mathcal{X}\left(\frac{4r}{r'}\right) &+ \mathcal{I}_{2\mu+2}(\lambda r) \Delta_{2\mu+2} \left( r^{\mu+1-d/2} \mathcal{X}\left(\frac{4r}{r'}\right) \right)\\
    &-2 \frac{\partial}{\partial r} \left( \mathcal{I}_{2\mu+2}(\lambda r) \right) \frac{\partial}{\partial r} \left( r^{\mu+1-d/2} \mathcal{X}\left(\frac{4r}{r'}\right)\right)\\
    &+ \frac{2\mu+2-d}{r}  \frac{\partial}{\partial r} \left( \mathcal{I}_{2\mu+2}(\lambda r) \right) r^{\mu+1-d/2} \mathcal{X}\left(\frac{4r}{r'}\right)\\
    &+ \frac{2\mu+2-d}{r} \mathcal{I}_{2\mu+2}(\lambda r)  \frac{\partial}{\partial r} \left( r^{\mu+1-d/2} \mathcal{X}\left(\frac{4r}{r'}\right) \right)\\
    &:= I+II+III+IV+V.
\end{align*}
Obviously, 
\begin{equation}
    I = -\lambda^2 \mathcal{I}_{2\mu+2}(\lambda r) r^{\mu+1-d/2} \mathcal{X}\left(\frac{4r}{r'}\right).
\end{equation}
As for $II$, a straightforward computation yields that
\begin{align*}
    II &= -(\mu+1-d/2)(3\mu+1-d/2) \mathcal{I}_{2\mu+2}(\lambda r) r^{\mu-1-d/2} \mathcal{X}\left(\frac{4r}{r'}\right)\\
    &+ \mathcal{I}_{2\mu+2}(\lambda r) r^{\mu+1-d/2} \Delta_{2\mu+2} \left( \mathcal{X}\left(\frac{4r}{r'}\right) \right)\\
    &- (2\mu +2-d) \mathcal{I}_{2\mu+2}(\lambda r) r^{\mu-d/2} \frac{\partial}{\partial r} \left( \mathcal{X}\left(\frac{4r}{r'}\right) \right):= II_1+II_2+II_3.
\end{align*}
While for $III$, we have
\begin{align*}
    III &= - \frac{2\mu+2-d}{r} r^{\mu+1-d/2} \mathcal{X}\left(\frac{4r}{r'}\right) \frac{\partial}{\partial r} \left( \mathcal{I}_{2\mu+2}(\lambda r) \right)\\
    &- 2 r^{\mu+1-d/2} \frac{\partial}{\partial r} \left( \mathcal{I}_{2\mu+2}(\lambda r) \right) \frac{\partial}{\partial r} \left(\mathcal{X}\left(\frac{4r}{r'}\right) \right):= III_1+III_2.
\end{align*}
Note that $III_1 + IV = 0$.

Last but not least, 
\begin{align*}
    V &= (2\mu+2-d) (\mu+1-d/2) \mathcal{I}_{2\mu+2}(\lambda r) r^{\mu-1-d/2} \mathcal{X}\left(\frac{4r}{r'}\right)\\
    &+ (2\mu+2-d) \mathcal{I}_{2\mu+2}(\lambda r) r^{\mu-d/2} \frac{\partial}{\partial r} \left( \mathcal{X}\left(\frac{4r}{r'}\right) \right):= V_1+V_2.
\end{align*}
Observe that $V_2+II_3=0$.

Therefore,
\begin{align*}
    I&+II+III+IV+V \\
    &= -\lambda^2 \mathcal{I}_{2\mu+2}(\lambda r) r^{\mu+1-d/2} \mathcal{X}\left(\frac{4r}{r'}\right) -(\mu+1-d/2)(3\mu+1-d/2) \mathcal{I}_{2\mu+2}(\lambda r) r^{\mu-1-d/2} \mathcal{X}\left(\frac{4r}{r'}\right)\\
    &+ \mathcal{I}_{2\mu+2}(\lambda r) r^{\mu+1-d/2} \Delta_{2\mu+2} \left( \mathcal{X}\left(\frac{4r}{r'}\right) \right) - 2 r^{\mu+1-d/2} \frac{\partial}{\partial r} \left( \mathcal{I}_{2\mu+2}(\lambda r) \right) \frac{\partial}{\partial r} \left(\mathcal{X}\left(\frac{4r}{r'}\right) \right)\\
    &+ (2\mu+2-d) (\mu+1-d/2) \mathcal{I}_{2\mu+2}(\lambda r) r^{\mu-1-d/2} \mathcal{X}\left(\frac{4r}{r'}\right)\\
    &= -\lambda^2 \mathcal{I}_{2\mu+2}(\lambda r) r^{\mu+1-d/2} \mathcal{X}\left(\frac{4r}{r'}\right) - \left(\mu^2 - \left(\frac{d-2}{2}\right)^2\right) \mathcal{I}_{2\mu+2}(\lambda r) r^{\mu-1-d/2} \mathcal{X}\left(\frac{4r}{r'}\right)\\
    &+ \mathcal{I}_{2\mu+2}(\lambda r) r^{\mu+1-d/2} \Delta_{2\mu+2} \left( \mathcal{X}\left(\frac{4r}{r'}\right) \right) - 2 r^{\mu+1-d/2} \frac{\partial}{\partial r} \left( \mathcal{I}_{2\mu+2}(\lambda r) \right) \frac{\partial}{\partial r} \left( \mathcal{X}\left(\frac{4r}{r'}\right) \right)\\
    &:= A+B+C+D.
\end{align*}
Next, it is clear to see that
\begin{align*}
    \lambda^{2-d+\mu} r'^{1-d/2} \cdot B \cdot K_{\mu}(\lambda r') u(y) \overline{u(y')} + \eqref{HG_G2_expand2} = 0.
\end{align*}
We finally end up with expression
\begin{align}\nonumber
    H_z &G_2^\lambda(z,z') = \sum_{j\ge 0} \lambda^{2-d+\mu_j} r'^{1-d/2} \cdot (A+C+D) \cdot K_{\mu_j}(\lambda r') u_j(y) \overline{u_j(y')}\\ \label{HG11}
    &= -\lambda^2 \sum_{j\ge 0} \lambda^{2-d} (rr')^{1-d/2} I_{\mu_j}(\lambda r) K_{\mu_j}(\lambda r') u_j(y) \overline{u_j(y')} \mathcal{X}\left(\frac{4r}{r'}\right)\\ \label{HG12}
    &+ \lambda^{2-d} (rr')^{1-d/2} \sum_{j\ge 0} I_{\mu_j}(\lambda r) K_{\mu_j}(\lambda r') u_j(y) \overline{u_j(y')} \Delta_{2\mu_j+2} \left(  \mathcal{X}\left(\frac{4r}{r'}\right) \right)\\ \label{HG13}
    &- 2 \lambda^{2-d} (rr')^{1-d/2} \sum_{j\ge 0} u_j(y) \overline{u_j(y')} K_{\mu_j}(\lambda r') \left( -\mu_j r^{-1} I_{\mu_j}(\lambda r) + \lambda I'_{\mu_j}(\lambda r) \right) \frac{\partial}{\partial r} \left(\mathcal{X}\left(\frac{4r}{r'}\right)\right).
\end{align}
We call \eqref{HG11} the main term and \eqref{HG12}, \eqref{HG13} the error terms.

Note that the main term \eqref{HG11} equals $-\lambda^2 G_2^\lambda$. So its estimate follows directly from Lemma~\ref{leHL1}, i.e.,
\begin{align*}
    |\eqref{HG11}| \lesssim \mathbf{1}_{r\le r'}\begin{cases}
        \lambda^{4-d} r^{1-d/2+\mu_0} r'^{1-d/2-\mu_0}, & \lambda \le r'^{-1},\\
        \lambda^{3-d/2+\mu_0-N} r^{1-d/2+\mu_0} r'^{-N}, & \lambda \ge r'^{-1},
    \end{cases}
\end{align*}
for all $N>0$.

While for the error terms \eqref{HG12} and \eqref{HG13}, we have first that
\begin{align*}
    \left|\Delta_{2\mu_j+2} \left(  \mathcal{X}\left(\frac{4r}{r'}\right) \right)\right| &\lesssim (1+\mu_j)r^{-2} \mathbf{1}_{r'/8 \le r\le r'/4},\\
    \left|\frac{\partial}{\partial r} \left(\mathcal{X}\left(\frac{4r}{r'}\right)\right) \right| &\lesssim r^{-1} \mathbf{1}_{r'/8 \le r\le r'/4}.
\end{align*}
It then follows by Lemma~\ref{leHL2} and Hörmander's $L^\infty$-estimate \cite{Hormander1} that 
\begin{align*}
    |\eqref{HG12}| \lesssim \lambda^{2-d} (rr')^{1-d/2} r^{-2} \mathbf{1}_{r'/8 \le r\le r'/4} \sum_{j=0}^\infty (\mu_j^{d-1} + \mu_j^d) \left(\frac{4r}{r'}\right)^{\mu_j} e^{-\lambda r'/4}\\
    \lesssim \lambda^{2-d} r^{-d} e^{-\lambda r'/4} \mathbf{1}_{r'/8 \le r\le r'/4}.
\end{align*}
To this end, we estimate \eqref{HG13} by a similar manner:
\begin{align*}
    |\eqref{HG13}| \lesssim \lambda^{2-d} (rr')^{1-d/2} \sum_{j\ge 0} \mu_j^{d} 4^{\mu_j} (\lambda r')^{-\mu_j} e^{-\lambda r'/2} r^{-1} \mathbf{1}_{r \sim r'} | r^{-1} (\lambda r)^{\mu_j} e^{\lambda r} + \lambda (\lambda r)^{\mu_j-1} e^{\lambda r}|\\
    \lesssim \lambda^{2-d} r^{2-d} e^{-\lambda r'/4} \sum_{j\ge 0} \mu_j^d \left(\frac{4r}{r'}\right)^{\mu_j} r^{-2} \mathbf{1}_{r\sim r'}\\
    \lesssim \lambda^{2-d} r^{-d} e^{-c\lambda r'} \mathbf{1}_{r'/8 \le r \le r'/4}
\end{align*}
as desired.

\end{proof}

Next, we prove the positive part of Theorem~\ref{thm2} (except endpoint estimates).

\begin{proposition}\label{propexha}
The reverse inequality:
\begin{align*}
    \|H^{1/2}f\|_p \le C \| \nabla_H f\|_p
\end{align*}
holds for
\begin{equation*}
    \frac{d}{\min \left( \frac{d+4}{2}+\mu_0, d\right)} < p \le p_1' = \frac{d}{\min \left( \frac{d}{2}+\mu_0, d\right)}.
\end{equation*}

\end{proposition}

\begin{proof}[Proof of Proposition~\ref{propexha}]
By duality, it suffices to show that for any $g\in C_c^\infty(M)$ such that $\|g\|_{p'}\le 1$,
\begin{align*}
    \left| \langle H^{1/2}f, g \rangle\right| \le C \|\nabla_H f\|_p \|g\|_{p'}.
\end{align*}
Note that the LHS above is bounded by
\begin{align*}
\left|\left\langle \nabla f,  (R_1+R_3) g \right\rangle \right| &+ \left| \left\langle \frac{V_0}{r}f, r^{-1} \int_0^\infty \lambda^{d-2} (G_1^\lambda + G_3^\lambda)g d\lambda \right\rangle \right|\\
&+ \left| \left\langle f, \frac{V_0}{r^2} \int_0^\infty \lambda^{d-2} G_2^\lambda g d\lambda \right\rangle + \left\langle \nabla f, \int_0^\infty \lambda^{d-2} \nabla G_2^\lambda g d\lambda \right\rangle \right|:= I+II+III.
\end{align*}
The first two terms can be handled easily. Observe first that by Hölder's inequality 
\begin{align*}
I \le  \|\nabla f\|_p \|(R_1+R_3)g\|_{p'} \le C \|\nabla_H f\|_p \|g\|_{p'},\quad \forall 1<p<p_0'.
\end{align*}
Second, Lemma~\ref{T1} and Lemma~\ref{T2T3} guarantee
\begin{align*}
II \le C \left\| \frac{f}{r} \right\|_p \| (T_1+T_3)g\|_{p'}\le C \|\nabla_H f\|_p \|g\|_{p'},\quad \forall 1<p<p_0'.
\end{align*}
Next, by integration by parts
\begin{align*}
III = \left|\left\langle f, \int_0^\infty \lambda^{d-2} HG_2^\lambda g d\lambda \right\rangle \right|.
\end{align*}
Define
\begin{align*}
    \mathcal{T}: u \mapsto r \int_0^\infty \lambda^{d-2} \left(HG_2^\lambda u\right) d\lambda.
\end{align*}

\begin{lemma}\label{le_EHA}
$\mathcal{T}$ is bounded in $L^p$ for 
\begin{align*}
    1< p < \frac{d}{\max \left(\frac{d-4}{2}-\mu_0, 0\right)}.
\end{align*}
\end{lemma}

\begin{proof}[Proof of Lemma~\ref{le_EHA}]
Note that
\begin{align*}
\mathcal{T}u(z) &= r \int_0^\infty \int_M H_z G_2^\lambda(z,z') u(z') dz' \lambda^{d-2} d\lambda\\
&= r \int_M \left(\int_0^\infty H_z G_2^\lambda(z,z') \lambda^{d-2} d\lambda \right) u(z') dz'.
\end{align*}
By Lemma~\ref{HG}, for $r\le r'/4$, the inner integral is bounded by
\begin{align*}
\int_0^{r'^{-1}} \lambda^{2} r^{1-d/2+\mu_0} r'^{1-d/2-\mu_0} d\lambda &+ \int_{r'^{-1}}^\infty \lambda^{1+d/2+\mu_0-N} r^{1-d/2+\mu_0} r'^{-N} d\lambda\\
&+ \int_0^\infty \lambda^{2-d} r^{-d} e^{-c\lambda r} \mathbf{1}_{r'/8 \le r \le r'/4} \lambda^{d-2} d\lambda\\
&\lesssim r^{1-d/2+\mu_0} r'^{-2-d/2-\mu_0} + r^{-1-d}\mathbf{1}_{r'/8 \le r \le r'/4}.
\end{align*}
Therefore, it is clear that 
\begin{align*}
    |\mathcal{T}u(z)| \lesssim r^{2-d/2+\mu_0} \int_{r\le r'} r'^{-2-d/2-\mu_0} |u(z')| dz' +  r^{-d} \int_{r\sim r'} |u(z')| dz' := \mathcal{T}_1u(z) + \mathcal{T}_2u(z).
\end{align*}
It follows by Hölder's inequality that the second term above is bounded by $C r^{-d/p} \|u\|_p$, which is of weak type $(p,p)$ for all $1<p<\infty$. Indeed, for all $\tau >0$, we have
\begin{align*}
    \textrm{vol} \left( \left\{z\in M; |\mathcal{T}_2u(z)|>\tau    \right\}  \right) &\lesssim \textrm{vol} \left( \left\{z\in M; r^{-d/p}\|u\|_p >\tau    \right\}  \right)\\
    &\lesssim \textrm{vol} \left( \left\{z\in M; r \le \left(\frac{\|u\|_p}{\tau} \right)^{\frac{p}{d}}   \right\}  \right)\\
    &\lesssim \left(\frac{\|u\|_p}{\tau} \right)^p.
\end{align*}
Thus by interpolation that $\mathcal{T}_2$ is bounded on $L^p$ for all $1<p<\infty$.

As for $\mathcal{T}_1$, we apply Lemma~\ref{leHL3} and conclude that $\mathcal{T}_1$ is bounded on $L^p$ for 
\begin{align*}
    1< p < \frac{d}{\max \left(\frac{d-4}{2}-\mu_0, 0\right)},
\end{align*}
completing the proof.
\end{proof}

Now, by Lemma~\ref{le_EHA}, we conclude that
\begin{align*}
    |III| = \left| \left\langle \frac{f}{r}, \mathcal{T}g  \right\rangle   \right| \le \left\| \frac{f}{r} \right\|_p \|\mathcal{T}g\|_{p'} \lesssim \|\nabla_H f\|_p \|g\|_{p'},
\end{align*}
for $p$ in the desired range, which completes the proof of Proposition~\ref{propexha}.

\end{proof}

\subsection{Interior harmonic annihilation}\label{sec4.2}

In this section we give an alternative proof of the positive part of Theorem~\ref{thm2} (i.e., Proposition~\ref{propexha}) and an expected endpoint estimate \eqref{RRend}. Motivated by the discussion in Section~\ref{sec1.3}, we develop an \emph{interior harmonic annihilation} scheme. 

Note that by spectral theory, for $f\in C_c^\infty(M)$,
\begin{align*}
    H^{1/2}f &= H^{-1/2} Hf = \int_0^\infty (H+\lambda^{2})^{-1} Hf d\lambda = \int_0^\infty \int_M (H+\lambda^2)^{-1}(z,z') Hf(z') dz' d\lambda\\
    &= \int_0^\infty \int_M H_{z'}(H+\lambda^2)^{-1}(z,z') f(z') dz' d\lambda = \sum_{i=1}^2 \mathcal{O}_{i1}(\nabla f) + \sum_{i=1}^2 \mathcal{O}_{i2}\left(V_0 f/r'\right) + \mathcal{S}(f/r'),
\end{align*}
where
\begin{align*}
    \mathcal{O}_{i1}: u \mapsto \int_M \left(\int_0^\infty \nabla_{z'} G_i^\lambda(z,z') \lambda^{d-2} d\lambda \right)  u(z') dz', \quad i=1,2, \quad \forall u\in C_c^\infty(T^*M)
\end{align*}

\begin{align*}
    \mathcal{O}_{i2}:  u \mapsto \int_M \left(\int_0^\infty r'^{-1} G_i^\lambda(z,z') \lambda^{d-2} d\lambda \right) u(z') dz',\quad i=1,2,\quad \forall u\in C_c^\infty(M)
\end{align*}
and
\begin{align*}
    \mathcal{S}:  u \mapsto \int_M \left( r' \int_0^\infty H_{z'} G_3^\lambda(z,z') \lambda^{d-2} d\lambda \right) u(z') dz', \quad \forall u\in C_c^\infty(M).
\end{align*}

\begin{lemma}\label{O11}
$\mathcal{O}_{11}$ is bounded in $L^p$ for all $1< p<\infty$.
\end{lemma}

\begin{proof}[Proof of Lemma~\ref{O11}]
Note that $G_1$ is self-adjoint. Indeed, 
\begin{align*}
    \overline{G_1(z',z)} &= \overline{(H+1)^{-1}(z',z)} \left( 1 - \mathcal{X}(4r'/r) - \mathcal{X}(4r/r') \right) \\
    &= (H+1)^{-1}(z,z') \left( 1 - \mathcal{X}(4r'/r) - \mathcal{X}(4r/r') \right) \\
    &= G_1(z,z').
\end{align*}
One infers that 
\begin{align*}
    \mathcal{O}_{11}^*g(z) = \nabla_z \int_M \int_0^\infty G_1(\lambda z, \lambda z') \lambda^{d-2} d\lambda g(z') dz' = R_1g(z),
\end{align*}
which is known to be $L^p$-bounded for all $1<p<\infty$. The result follows by duality. 

\end{proof}

\begin{lemma}\label{O21}
$\mathcal{O}_{21}$ is bounded on $L^p$ for all $1<p<p_0'$. In addition, at the endpoint $p=p_0'$, $\mathcal{O}_{21}$ is bounded on $L^{p_0',1} \to L^{p_0', \infty}$.
\end{lemma}

\begin{proof}[Proof of Lemma~\ref{O21}]
Recall Remark~\ref{remark1}
\begin{align*}
    |\nabla_{z'} G_2^\lambda(z,z')| \lesssim \begin{cases}
        \lambda^{2-d} r^{1-d/2+\mu_0} r'^{-d/2-\mu_0}, & \lambda \le r'^{-1},\\
        \lambda^{1-d/2+\mu_0-N} r^{1-d/2+\mu_0} r'^{-N-1}, & \lambda \ge r'^{-1},
    \end{cases}
\end{align*}
for all $N>0$.

Hence, by choosing $N$ large enough, we yield
\begin{align*}
    |\mathcal{O}_{21}(z,z')|\lesssim \int_0^{r'^{-1}} r^{1-d/2+\mu_0} r'^{-d/2-\mu_0} d\lambda + \int_{r'^{-1}}^\infty  \lambda^{d/2-1+\mu_0-N} r^{1-d/2+\mu_0} r'^{-N-1} d\lambda\\
    \lesssim r^{1-d/2+\mu_0} r'^{-1-d/2-\mu_0}.
\end{align*}
Apply Lemma~\ref{leHL3} with $\alpha = d/2-1-\mu_0$ to conclude that $\mathcal{O}_{21}$ is bounded on $L^p$ for
\begin{align*}
    1<p<\frac{d}{\max \left(\frac{d-2}{2}-\mu_0, 0\right)} = p_0'.
\end{align*}
Particularly, if $\frac{d-2}{2}>\mu_0$, we have endpoint estimate:
\begin{align*}
    |\mathcal{O}_{21}u(z)|&\lesssim  r^{1-d/2+\mu_0} \int_{r\le r'} r'^{-1-d/2-\mu_0} |u(z')| dz'\\
    &\lesssim  r^{1-d/2+\mu_0} \|u\|_{(p_0',1)} \| g_r\|_{(p_0,\infty)},
\end{align*}
where $g_r(r') = r'^{-1-d/2-\mu_0} \mathbf{1}_{r\le r'}(r')$.

Note that
\begin{equation*}
    \| g_r\|_{(p_0,\infty)}^{p_0} = \sup_{0<\lambda < r^{-1-d/2-\mu_0}} \lambda^{p_0} d_{g_r}(\lambda) \lesssim \sup_{0<\lambda < r^{-1-d/2-\mu_0}} \lambda^{p_0} \lambda^{-\frac{2d}{d+2+2\mu_0}} \sim 1.
\end{equation*}
The result follows immediately since $r^{1-d/2+\mu_0} \in L^{p_0', \infty}$.

\end{proof}

\begin{lemma}\label{O12}
$\mathcal{O}_{12}$ is bounded on $L^p$ for all $1 < p<\infty$.
\end{lemma}

\begin{proof}[Proof of Lemma~\ref{O12}]
Note that
\begin{align*}
    \int_0^\infty r'^{-1} G_1^\lambda(z,z') \lambda^{d-2} d\lambda &\lesssim r'^{-1} \int_0^{d(z,z')^{-1}} d(z,z')^{2-d} d\lambda \\
    &+ r'^{-1} \int_{d(z,z')^{-1}}^\infty \lambda^{d-2-N} d(z,z')^{-N} d\lambda \lesssim r^{-1} d(z,z')^{1-d},
\end{align*}
where we replace $r'^{-1}$ with $r^{-1}$ in the last inequality since $G_1$ is supported near the diagonal $\{r/8 \le r' \le 8r\}$. Hence,
\begin{align*}
    |\mathcal{O}_{12}u(z)| \lesssim r^{-1} \int_{\{r/8\le r' \le 8r\}} \frac{|u(z')|}{d(z,z')^{d-1}} dz'.
\end{align*}
The result follows by the proof of Lemma~\ref{T1}.

\end{proof}

\begin{lemma}\label{O22}
$\mathcal{O}_{22}$ is bounded on $L^p$ for all $1 < p<p_0'$. In addition, at the endpoint $p=p_0'$, $\mathcal{O}_{22}$ is bounded on $L^{p_0',1} \to L^{p_0', \infty}$.
\end{lemma}

\begin{proof}[Proof of Lemma~\ref{O22}]
A straightforward estimate yields that for $r\le 4r'$,
\begin{align*}
    |\mathcal{O}_{22}(z,z')|\lesssim r^{1-d/2+\mu_0} r'^{-1-d/2-\mu_0}
\end{align*}
It then follows by Lemma~\ref{leHL3} that $\mathcal{O}_{22}$ is $L^p$-bounded for 
\begin{equation*}
    1<p<\frac{d}{\max \left(\frac{d-2}{2}-\mu_0, 0\right)} = p_0'.
\end{equation*}
The endpoint result follows by the same argument from the proof of Lemma~\ref{O21}.
\end{proof}

\begin{lemma}\label{S}
$\mathcal{S}$ is bounded on $L^p$ for
\begin{equation*}
    p>\frac{d}{\min\left(\frac{d+4}{2}+\mu_0, d \right)}.
\end{equation*}
In addition, at the endpoint $p=\frac{d}{\min\left(\frac{d+4}{2}+\mu_0, d \right)}$, $\mathcal{S}$ is bounded on $L^{p,1} \to L^{p, \infty}$.
\end{lemma}

\begin{proof}[Proof of Lemma~\ref{S}]
By Lemma~\ref{HG}, one deduces kernel estimate
\begin{align*}
    |\mathcal{S}(z,z')| &\lesssim r' \mathbf{1}_{r'\le r/4} \int_0^{r^{-1}} \lambda^2 r^{1-d/2-\mu_0} r'^{1-d/2+\mu_0} d\lambda\\
    &+ r' \mathbf{1}_{r'\le r/4} \int_{r^{-1}}^\infty \lambda^{1+d/2+\mu_0-N} r^{-N} r'^{1-d/2+\mu_0} d\lambda\\
    &+ r'^{1-d} \mathbf{1}_{r\sim r'} \int_0^\infty e^{-c\lambda r} d\lambda \\
    &\lesssim r^{-2-d/2-\mu_0} r'^{2-d/2+\mu_0} \mathbf{1}_{r'\le r/4} +   r^{-d} \mathbf{1}_{r\sim r'}.
\end{align*}
Now, the second term above can be handled by Hölder's inequality and a weak type argument as in the proof of Lemma~\ref{le_EHA}, showing that it acts as a bounded operator on $L^p$ for all $1<p<\infty$. While for the first term, one applies Lemma~\ref{leHL3} once more to obtain that $\mathcal{S}$ is $L^p$-bounded for $p$ in the desired range.

In particular, for $p=\frac{d}{\min\left(\frac{d+4}{2}+\mu_0, d \right)}$, we have (the harmless error term $O(r^{-d} \mathbf{1}_{r\sim r'})$ has been omitted)
\begin{align*}
    |\mathcal{S}u(z)| &\lesssim r^{-2-d/2-\mu_0} \int_{r'\le r} r'^{2-d/2+\mu_0} |u(z')| dz'\\
    &\lesssim r^{-2-d/2-\mu_0} \|u\|_{(p,1)} \| g_r\|_{(p',\infty)},
\end{align*}
where $g_r(r') = r'^{2-d/2+\mu_0} \mathbf{1}_{r'\le r}$.

Easily, one computes
\begin{align*}
    \| g_r\|_{(p',\infty)} = \begin{cases}
        \|g_r\|_\infty = r^{2-d/2+\mu_0}, & \mu_0 \ge \frac{d-4}{2},\\
        \max \left( \sup_{0<\lambda<r^{2-d/2+\mu_0}} \lambda r^{d/p' }, \sup_{\lambda \ge r^{2-d/2+\mu_0}} \lambda \lambda^{\frac{1}{(2-d/2+\mu_0)p'}}  \right) = 1, & \mu_0 < \frac{d-4}{2}.
    \end{cases}
\end{align*}
Therefore, if $\mu_0\ge \frac{d-4}{2}$, then $|\mathcal{S}u(z)|\lesssim r^{-d}\|u\|_1$ and hence $\mathcal{S}$ is of weak type $(1,1)$ since $r^{-d}\in L^{1,\infty}$. While if $\mu_0<\frac{d-4}{2}$, then $|\mathcal{S}u(z)|\lesssim r^{-2-d/2-\mu_0}\|u\|_{(p,1)}$ and the result follows because $r^{-2-d/2-\mu_0}\in L^{p,\infty}$.

\end{proof}

We end this subsection by showing: 

\begin{proposition}\label{positive}
The reverse inequality:
\begin{align*}
    \|H^{1/2}f\|_p \le C \| \nabla_H f\|_p
\end{align*}
holds for
\begin{equation*}
    \frac{d}{\min \left( \frac{d+4}{2}+\mu_0, d\right)} < p < p_0'.
\end{equation*}
In addition, we have endpoint estimate
\begin{equation*}
    \|H^{1/2}f\|_{(p,\infty)} \le C \| \nabla_H f\|_{(p,1)}
\end{equation*}
for $p = \frac{d}{\min \left( \frac{d+4}{2}+\mu_0, d\right)}$, and if $p_0'<\infty$, then
\begin{equation*}
    \|H^{1/2}f\|_{(p_0',\infty)} \le C \| \nabla_H f\|_{(p_0',1)}.
\end{equation*}
\end{proposition}

\begin{proof}[Proof of Proposition~\ref{positive}]
Now, by Lemma~\ref{O11}-\ref{S}, we end up with estimate
\begin{align*}
    \|H^{1/2}f\|_p \lesssim \sum_{i=1}^2 \|\mathcal{O}_{i1}(\nabla f)\|_p + \sum_{i=1}^2 \| \mathcal{O}_{i2}(V_0f/r')\|_p + \|\mathcal{S}(f/r')\|_p \lesssim \|\nabla f\|_p + \left\| \frac{f}{r'} \right\|_p \lesssim \|\nabla_H f\|_p
\end{align*}
for all
\begin{equation*}
    \frac{d}{\min \left(d, \frac{d+4}{2}+\mu_0\right)} < p < \frac{d}{\max \left(0, \frac{d-2}{2}-\mu_0 \right)}. 
\end{equation*}
While for the endpoints $p=p_0', \frac{d}{\min\left( \frac{d+4}{2}+\mu_0, d \right)}$, a similar argument works by the endpoint estimates from Lemma~\ref{O21}, Lemma~\ref{O22}, Lemma~\ref{S} and interpolation. This completes the proof of Proposition~\ref{positive} and hence the positive part of Theorem~\ref{thm2}.
\end{proof}

\subsection{Unboundedness of the reverse inequality}\label{sec4.3}

By the argument in the previous section, to prove the sharpness of the lower threshold 
\[
p=\frac{d}{\min\!\left(d,\frac{d+4}{2}+\mu_0\right)},
\]
it suffices to analyze the sub-operator \(\mathcal S\) coming from the \(G_3\)-piece. 
Equivalently, we study the kernel of
\begin{equation*}
    \mathcal U_1 f(z)
    := \int_M \Biggl( \int_0^\infty H_{z'}\,G_3^\lambda(z,z')\,\lambda^{d-2}\,d\lambda \Biggr) f(z')\,d\mu(z'),
\end{equation*}
and construct \(f\) with \(\|\nabla_H f\|_{p}<\infty\) such that \(|\mathcal U_1 f(z)|=+\infty\) a.e. whenever 
\(p \le \dfrac{d}{\min\!\left(d,\frac{d+4}{2}+\mu_0\right)}\).

For the upper threshold \(p=p_0'\), it remains to examine the operator
\(\mathcal O_{21}(\nabla f)+\mathcal O_{22}(V_0 f/r')\), i.e.
\begin{equation*}
   \mathcal U_2 f(z)
   := \mathcal O_{21}(\nabla f)(z)+\mathcal O_{22}(V_0 f/r')(z)
   = \int_M \Biggl( \int_0^\infty H_{z'}\,G_2^\lambda(z,z')\,\lambda^{d-2}\,d\lambda \Biggr) f(z')\,d\mu(z').
\end{equation*}
We analyze its kernel and produce a counterexample \(f\) with \(\|\nabla_H f\|_{p}<\infty\) such that 
\(|\mathcal U_2 f(z)|=+\infty\) a.e. for all \(p\ge p_0'\).

We first show that the lower threshold cannot be improved in the strong \(L^p\) sense.

\begin{lemma}\label{lowerthrshold}
The reverse inequality
\begin{align*}
    \| H^{1/2}f\|_p \lesssim  \| \nabla f\|_p + \| f/r\|_p
\end{align*}
does not hold for $1<p\le \frac{d}{\min \left(d,\frac{d+4}{2}+\mu_0 \right)}$.
\end{lemma}

\begin{proof}[Proof of Lemma~\ref{lowerthrshold}]
We only consider the case $d>\frac{d+4}{2}+\mu_0$ or there is nothing to prove. We make the following series of reductions. By Lemma~\ref{HG}, the kernel $H_z G_2^\lambda(z,z')$ has been calculated explicitly. Therefore, it follows by symmetry, one has
\begin{align}\label{HGz1}
    H_{z'}&G_3^\lambda(z,z') = -\lambda^2 \sum_{j\ge 0} \lambda^{2-d} (rr')^{1-d/2} I_{\mu_j}(\lambda r') K_{\mu_j}(\lambda r) u_j(y') \overline{u_j(y)} \mathcal{X}\left(\frac{4r'}{r}\right)\\ \label{HGz2}
    &+ \lambda^{2-d} (rr')^{1-d/2} \sum_{j\ge 0} I_{\mu_j}(\lambda r') K_{\mu_j}(\lambda r) u_j(y') \overline{u_j(y)} \Delta_{2\mu_j+2} \left(  \mathcal{X}\left(\frac{4r'}{r}\right) \right)\\ \label{HGz3}
    &- 2 \lambda^{2-d} (rr')^{1-d/2} \sum_{j\ge 0} u_j(y') \overline{u_j(y)} K_{\mu_j}(\lambda r) \left( -\mu_j r'^{-1} I_{\mu_j}(\lambda r') + \lambda I'_{\mu_j}(\lambda r') \right) \frac{\partial}{\partial r'} \left(\mathcal{X}\left(\frac{4r'}{r}\right)\right).
\end{align}

Again, we call \eqref{HGz2}, \eqref{HGz3} the error terms and \eqref{HGz1} the main term. Next, by a parallel estimate, the error terms has upper bound $O(\lambda^{2-d} r'^{-d} e^{-c\lambda r}) \mathbf{1}_{r/8 \le r' \le r/4}$, which contributes $O(r'^{-1-d}) \mathbf{1}_{r \sim r'}$ in the kernel $\mathcal{U}_1(z,z')$. Thus, it follows by a similar argument as in Lemma~\ref{le_EHA} that the operator
\begin{equation*}
    u \mapsto \int_M  \left( \int_0^\infty r' \left( \eqref{HGz2} + \eqref{HGz3} \right) \lambda^{d-2} d\lambda \right) u(z') dz'
\end{equation*}
acts boundedly on $L^p$ for all $1<p<\infty$. Therefore, Lemma~\ref{Hardy} implies that
\begin{align*}
    \left\| \int_M  \left( \int_0^\infty \left( \eqref{HGz2} + \eqref{HGz3} \right) \lambda^{d-2} d\lambda \right) f(z') dz' \right\|_p \le C \left\| \frac{f}{r} \right\|_p,\quad p\in (1,\infty).
\end{align*}
Hence, to prove Lemma~\ref{lowerthrshold}, it is enough to consider the main part (i.e., \eqref{HGz1}). Note that each term in the sum \eqref{HGz1} is $O(\lambda^{4-d} r^{1-d/2-\mu_j} r'^{1-d/2+\mu_j} e^{-\lambda r}) \mathbf{1}_{r'\le r/4}$; see Lemma~\ref{HG} for the parallel estimate. Therefore, each term of \eqref{HGz1} contributes $O( r^{-2-d/2-\mu_j} r'^{1-d/2+\mu_j}) \mathbf{1}_{r'\le r/4}$ in the kernel of $\mathcal{U}_1$, which, by Lemma~\ref{leHL3}, acts as a bounded operator on $L^p$ in the sense
\begin{align*}
    \left\| \int_M  \left( \int_0^\infty \left( \textrm{each term of}\quad \eqref{HGz1} \right) \lambda^{d-2} d\lambda \right) f(z') dz' \right\|_p \le C \left\| \frac{f}{r} \right\|_p.
\end{align*}
if
\begin{equation*}
    p > \frac{d}{\min\left( d, \frac{d+4}{2} + \mu_j \right)}.
\end{equation*}

Now, since $Y$ is compact, it follows by Weyl's law that $\mu_j \to \infty$ as $j\to \infty$, and there exists some $J>0$ such that for all $j>J$, one has $d\le \frac{d+4}{2}+\mu_j$. This observation allows us reducing the problem to find a function $f$ such that $\|\nabla f\|_p + \| f/r\|_p <\infty$ but 
\begin{align*}
\mathcal{H}f(z) := \sum_{j=0}^{J} \int_M (rr')^{1-d/2} \mathcal{X}\left(\frac{4r'}{r}\right) \left(\int_0^\infty I_{\mu_j}(\lambda r') K_{\mu_j}(\lambda r) \lambda^2 d\lambda\right) u_j(y') \overline{u_j(y)} f(r',y') r'^{d-1} dr'dy'
\end{align*}
blows up a.e.

Since $\frac{r'}{r}\le \frac{1}{4}<1$, the formula from \cite[Page~684, 6.576.5]{GR} guarantees that
\begin{equation*}
    \int_0^\infty I_{\mu_j}(\lambda r') K_{\mu_j}(\lambda r) \lambda^2 d\lambda = C_{\mu_j} r^{-3} \left(\frac{r'}{r}\right)^{\mu_j} {}_2 F_1\left( \mu_j + \frac{3}{2}, \frac{3}{2}; \mu_j+1; \left(\frac{r'}{r}\right)^2 \right),
\end{equation*}
where ${}_2F_1$ refers to the Gaussian hypergeometric function. Observe that $0<(r'/r)^2\le 1/16$. The monotonicity confirms that $1= {}_2F_1(a,b;c;0) \le {}_2F_1(a,b;c;(r'/r)^2)\le {}_2F_1(a,b;c;1/16) = C$ if $a,b,c>0$. Consider function $f(r',y') = \overline{u_0(y')} r'^{-1-d/2-\mu_0} \left(1+|\log{r'}|\right)^{-s} \eta(r')$, where $1/p < s \le 1$ and $\eta\in C_c^\infty((0,1])$ is a positive smooth cut-off function, which equals $1$ in $(0,1/2]$. By orthogonality, for $i\ne j$, we have $\langle u_j, u_i \rangle = 0$. Thus 
\begin{equation*}
|\mathcal{H}f(z)| \ge C r^{-2-d/2-\mu_0} |\overline{u_0(y)}| \|u_0\|_2^2 \int_0^\infty \mathcal{X}\left(\frac{4r'}{r}\right) \left(1+|\log{r'}|\right)^{-s} \eta(r') \frac{dr'}{r'}.
\end{equation*}
Now, for a.e. $z=(r,y)$, the above integral is bounded below by (set $\sigma = \min \left(\frac{1}{2}, \frac{r}{8} \right)$ and recall that $\mathcal{X}(t)=1$ for $0<t\le 1/2$)
\begin{align*}
    \int_0^\sigma \left(1+|\log{r'}|\right)^{-s} \frac{dr'}{r'} = \int_{|\log{\sigma}|}^\infty (1+t)^{-s} dt = \infty
\end{align*}
since $s\le 1$, which implies that $\mathcal{H}f(z)$ blows up for a.e. $z=(r,y)$.

Next, we check that indeed $\|\nabla f\|_p + \|f/r\|_p <\infty$ for $p\le \frac{d}{\frac{d+4}{2}+\mu_0}$. Note that by Hörmander's $L^\infty$-estimate,
\begin{align*}
    |\nabla f(z)| \lesssim \left|\frac{d}{dr} \left( \eta(r) r^{-1-d/2-\mu_0} (1+|\log{r}|)^{-s} \right) \right| + \left| 
\eta(r) r^{-2-d/2-\mu_0} (1+|\log{r}|)^{-s} \right|  \\
\lesssim r^{-2-d/2-\mu_0}(1+|\log{r}|)^{-s} \mathbf{1}_{(0,1]}(r).
\end{align*}
We compute
\begin{align*}
    \|\nabla f\|_p^p &\lesssim \int_0^1 r^{-p\left( 2+\frac{d}{2}+\mu_0 \right)+d-1} (1+|\log{r}|)^{-sp} dr\\
    &\sim \int_0^\infty e^{-t\left[ d - p\left( \frac{d+4}{2} + \mu_0 \right)\right]} (1+t)^{-sp} dt.
\end{align*}
Obviously, for $1\le p < \frac{d}{\frac{d+4}{2}+\mu_0 }$, the integral converges absolutely. As for the special case $p = \frac{d}{\frac{d+4}{2}+\mu_0 }$, the integral remains finite since $s>1/p$. The estimate for $\|f/r\|_p$ is similar. The proof is now complete.
\end{proof}

\begin{remark}\label{remark3}
A direct computation shows that, for any \(1<q\le \infty\) and any \(s\in(1/q,1]\), our counterexample:
\[
f(z)=\overline{u_0(y)}\, r^{-1-\frac{d}{2}-\mu_0}\,\bigl(1+|\log r|\bigr)^{-s}\,\eta(r)
\]
satisfies property: $|\nabla_H f|\in L^{p,q}$ with $p=\frac{d}{\frac{d+4}{2}+\mu_0}$.

This establishes the sharpness of the Lorentz-type endpoint estimate:
\[
\|H^{1/2}f\|_{(p,\infty)} \le C\,\|\nabla_H f\|_{(p,q)},
\qquad p=\frac{d}{\frac{d+4}{2}+\mu_0},
\]
in the sense that the above inequality holds for $q=1$ but fails for every $q>1$.
\end{remark}

Next, we verify that the upper threshold is also optimal in the sense of $L^p$ norm.

\begin{lemma}\label{upperthrshold}
The reverse inequality
\begin{align*}
    \| H^{1/2}f\|_p \lesssim \| \nabla f\|_p + \| f/r\|_p
\end{align*}
does not hold for $p_0'= \frac{d}{\max\left(0, \frac{d-2}{2}-\mu_0\right)} \le p<\infty$.
\end{lemma}

\begin{proof}[Proof of Lemma~\ref{upperthrshold}]
It suffices to show the case $\frac{d-2}{2}> \mu_0$, i.e., $p_0' = \frac{d}{\frac{d-2}{2}-\mu_0}$. We follow the same idea as in the proof of Lemma~\ref{lowerthrshold}. We compute the kernel of $H_{z'} G_2^\lambda(z,z')$ explicitly. Set $\mathcal{K}_\alpha(x) = x^{1-\alpha/2} K_{|\alpha/2-1|}(x)$. Then, similar to $\mathcal{I}_\alpha$, one has $\Delta_\alpha \mathcal{K}_\alpha(\lambda x) = - \lambda^2 \mathcal{K}_\alpha(\lambda x)$ for all $x>0$. Replace $u_j$ with $u$ and $\mu_j$ with $\mu$ for simplicity. Each term in the sum \eqref{HG_G2} can be written as
\begin{equation*}
    \lambda^{2-d+\mu} r^{1-\frac{d}{2}} r'^{1-\frac{d}{2}+\mu} u(y) \overline{u(y')} I_\mu(\lambda r) \mathcal{K}_{2\mu+2}(\lambda r') \mathcal{X}\left(\frac{4r}{r'}\right).
\end{equation*}
It follows by the same strategy as in Lemma~\ref{HG} that 
\begin{align}\label{uppermain}
    H_{z'}&G_2^\lambda(z,z') = - \lambda^{4-d} \sum_{j=0}^\infty (rr')^{1-\frac{d}{2}} I_{\mu_j}(\lambda r) K_{\mu_j}(\lambda r') u_j(y) \overline{u_j(y')} \mathcal{X}\left( \frac{4r}{r'} \right)\\ \label{uppererror1}
    &+ \lambda^{2-d} (rr')^{1-d/2} \sum_{j\ge 0} I_{\mu_j}(\lambda r) K_{\mu_j}(\lambda r') u_j(y) \overline{u_j(y')} \Delta_{2\mu_j+2, r'} \left(  \mathcal{X}\left(\frac{4r}{r'}\right) \right)\\ \label{uppererror2}
    &- 2 \lambda^{2-d} (rr')^{1-d/2} \sum_{j\ge 0} u_j(y) \overline{u_j(y')} I_{\mu_j}(\lambda r) \left( -\mu_j r'^{-1} K_{\mu_j}(\lambda r') + \lambda K'_{\mu_j}(\lambda r') \right) \frac{\partial}{\partial r'} \left(\mathcal{X}\left(\frac{4r}{r'}\right)\right),
\end{align}
where we add a footscript $r'$ in $\Delta_{2\mu_j+2, r'}$ to emphasize that the 'Laplacian' is acting on $r'$ variable. One checks easily
\begin{align*}
    \left| \Delta_{2\mu_j+2, r'} \left(  \mathcal{X}\left(\frac{4r}{r'}\right) \right) \right| &\lesssim (2\mu_j+1) r^{-2} \mathbf{1}_{r\sim r'},\\
    \left| \frac{\partial}{\partial r'} \left(\mathcal{X}\left(\frac{4r}{r'}\right)\right) \right| &\lesssim r^{-1} \mathbf{1}_{r\sim r'}.
\end{align*}
It is then plain that the error terms have uniform upper bound
\begin{equation*}
    |\eqref{uppererror1} + \eqref{uppererror2}| \lesssim \lambda^{2-d} r^{-d} e^{-\lambda r'/4} \mathbf{1}_{r\sim r'}.
\end{equation*}
Hence, a similar argument as in the proof of Lemma~\ref{le_EHA} (see also the argument below \eqref{HGz3}) yields that
\begin{align*}
    \left\| \int_M  \left( \int_0^\infty \left( \eqref{uppererror1} + \eqref{uppererror2} \right) \lambda^{d-2} d\lambda \right) f(z') dz' \right\|_p \le C \left\|\frac{f}{r} \right\|_p,\quad \forall p\in (1,\infty).
\end{align*}
Therefore, we may focus on the main term \eqref{uppermain}. Note that each term in the sum \eqref{uppermain} contributes $O( \lambda^{4-d} r^{1-d/2+\mu_j} r'^{1-d/2-\mu_j} e^{-\lambda r/4}) \mathbf{1}_{r\le r'/4}$ to the kernel $H_{z'}G_2^\lambda(z,z')$, and hence contributes $O(r^{-2-d/2+\mu_j} r'^{1-d/2-\mu_j}) \mathbf{1}_{r\le r'/4}$ to $\mathcal{U}_2(z,z')$. So, Lemma~\ref{leHL3} tells us
\begin{align}\label{inthesense}
    \left\| \int_M  \left( \int_0^\infty \left( \textrm{each term of}\quad \eqref{uppermain} \right) \lambda^{d-2} d\lambda \right) f(z') dz' \right\|_p \le C \left\| \frac{f}{r} \right\|_p.
\end{align}
if
\begin{equation*}
    p < \frac{d}{\max \left( 0, \frac{d+4}{2} - \mu_j \right)}.
\end{equation*}
Thus, by extracting finitely many terms from \eqref{uppermain}, say $\Tilde{J}$ terms, the remainder always acts boundedly in $L^p$ for all $1<p<\infty$ (in the sense \eqref{inthesense}). To this end, one only needs to show that there exists a function $f$ such that $\|\nabla f\|_p + \|f/r\|_p<\infty$ with $\Tilde{\mathcal{H}}f$ blows up a.e., where
\begin{equation*}
    \Tilde{\mathcal{H}}f(z) := \sum_{j=0}^{\Tilde{J}} \int_M (rr')^{1-d/2} \mathcal{X}\left(\frac{4r}{r'}\right) \left(\int_0^\infty I_{\mu_j}(\lambda r) K_{\mu_j}(\lambda r') \lambda^2 d\lambda\right) u_j(y) \overline{u_j(y')} f(r',y') r'^{d-1} dr'dy'.
\end{equation*}
By \cite[Page~684, 6.576.5]{GR}, for $\frac{r}{r'}\le \frac{1}{4} <1$,
\begin{equation*}
    \int_0^\infty I_{\mu_j}(\lambda r) K_{\mu_j}(\lambda r') \lambda^2 d\lambda = C_{\mu_j} r'^{-3} \left(\frac{r}{r'}\right)^{\mu_j} {}_2 F_1\left( \mu_j + \frac{3}{2}, \frac{3}{2}; \mu_j+1; \left(\frac{r}{r'}\right)^2 \right)
\end{equation*}
Let $\phi \in C_c^\infty([0,2])$ such that $0\le \phi \le 1$ and $\phi(t) = 1$ for $t\le 1$. For $R>0$ large, consider function $f_R(r',y') = u_0(y') r'^{2-d/2+\mu_0} (1+|\log{r'}|)^{-s} \Tilde{\eta}(r') \phi(r'/R) $ with $1/p<s\le 1$, where $\Tilde{\eta} \in C^\infty(\mathbb{R})$, non-negative, supports in $[2,\infty)$ and equals $1$ in $[4,\infty)$. It follows that
\begin{align*}
    &\Tilde{\mathcal{H}}f_R(z)\\
    &= C_{\mu_0} r^{1-\frac{d}{2}+\mu_0} u_0(y) \|u_0\|_2^2 \int_0^\infty \mathcal{X}\left(\frac{4r}{r'}\right) {}_2 F_1\left( \mu_0 + \frac{3}{2}, \frac{3}{2}; \mu_0+1; \left(\frac{r}{r'}\right)^2 \right)   \frac{\Tilde{\eta}(r') \phi\left(\frac{r'}{R}\right)}{(1+|\log{r'}|)^{s}}  \frac{dr'}{r'}.
\end{align*}
Next, since $0<\left( \frac{r}{r'} \right)^2\le \frac{1}{16}$, by the monotonicity of ${}_2 F_1$ (when all entries are positive), for a.e. $(r,y)\in M$, the integral above is bounded below by (set $\delta = \max( 8r, 4 )$)
\begin{equation*}
    \int_{\delta}^R (1+|\log{t}|)^{-s} \frac{dt}{t} = \int_{\log{\delta}}^{\log{R}} (1+x)^{-s} dx \to \infty, \quad \textrm{as}\quad R\to \infty,
\end{equation*}
since $s\le 1$.

However, we compute for all $R>0$,
\begin{equation*}
    |\nabla f_R| + |f_R/r'| \lesssim r'^{1-d/2+\mu_0} (1+|\log{r'}|)^{-s} \mathbf{1}_{r'\ge 2},
\end{equation*}
and whence
\begin{align*}
    \| \nabla f_R\|_p^p + \|f_R/r\|_p^p &\lesssim \int_1^\infty t^{p(1-d/2+\mu_0)} (1+|\log{t}|)^{-s} t^{d-1} dt\\
    &\sim \int_0^\infty e^{-\left[ p\left( \frac{d-2}{2} - \mu_0 \right) - d \right] t} (1+t)^{-sp} dt < \infty 
\end{align*}
for all
\begin{equation*}
    p > \frac{d}{\frac{d-2}{2} - \mu_0},
\end{equation*}
or $p= \frac{d}{\frac{d-2}{2} - \mu_0}$ and $s > 1/p$. The proof of Lemma~\ref{upperthrshold} is now complete.

\end{proof}

\begin{remark}\label{remark2}
A routine computation shows that, for each $1<q\le \infty$, at the endpoint $p=p_0'$ the counterexample
\begin{equation*}
    f_R(z) = u_0(y) r^{2-d/2+\mu_0} (1+|\log{r}|)^{-s} \Tilde{\eta}(r) \phi(r/R),\quad 1/q < s \le 1,\quad R>0
\end{equation*}
satisfies the uniform bound
\[
    \sup_{R\ge 10}\;\|\nabla_H f_R\|_{(p_0',q)} \le C,
\]
where $C$ is independent of $R$.

Consequently, the Lorentz-type endpoint estimate
\begin{equation*}
    \| H^{1/2} f \|_{(p_0',\infty)} \le C \, \|\nabla_H f\|_{(p_0',q)}
\end{equation*}
is sharp: it holds with $q=1$, but fails for every $q>1$.

\end{remark}

\appendix

\section{Asymptotic calculus}\label{sec5}

The asymptotic expansions obtained in this appendix fit naturally within the 
framework of microlocal analysis on manifolds with boundary, where the 
behaviour of distributional kernels is described by polyhomogeneous conormal 
structures on the iterated blow-up of the double space. 
Following the foundational development of the $b$- and scattering calculi 
by Melrose \cite{Melrose1,Melrose2}, one resolves the 
geometric singularities of resolvent-type kernels by working on the 
blown-up space $M_{b,sc}^2$ and tracking their index sets at the boundary 
hypersurfaces. 
Subsequent refinements by Hassell-Melrose-Vasy and Guillarmou-Hassell 
\cite{HMV,GH1,GH2} provide a precise 
microlocal description of the resolvent and spectral measure near radial 
points and at the scattering faces. 
The explicit computations carried out below should therefore be viewed as a 
concrete model of the polyhomogeneous structure predicted by this analytic 
framework, specialized to the Riesz potential $H^{-1/2}(z,z')$ near the 
right and left boundary faces.

We begin by recalling \eqref{formula1}.  
For $4r \le r'$, one has
\begin{equation*}
    (H+\lambda^2)^{-1}(z,z')
    = (rr')^{1-d/2}
      \sum_{j\ge 0} u_j(y)\, \overline{u_j(y')}
      I_{\mu_j}(\lambda r)\, K_{\mu_j}(\lambda r').
\end{equation*}
Therefore, integrating in $\lambda$ and using the classical identity 
involving the Gaussian hypergeometric function, we obtain
\begin{align*}
    &H^{-1/2}(z,z')
    \sim (rr')^{1-d/2}
        \sum_{j\ge 0} u_j(y)\, \overline{u_j(y')}
        \int_0^\infty 
            I_{\mu_j}(\lambda r)\, K_{\mu_j}(\lambda r')\, d\lambda \\
    &\sim \sum_{j=0}^\infty
        (rr')^{1-d/2}
        u_j(y)\, \overline{u_j(y')}
        \frac{\Gamma(1/2+\mu_j)\Gamma(1/2)}{2\Gamma(\mu_j+1)}
        \left(\frac{r}{r'}\right)^{\mu_j}
        r'^{-1}\,
        {}_2F_1\!\left(
            \frac12+\mu_j,\frac12;\mu_j+1;\frac{r^2}{r'^2}
        \right).
\end{align*}
We now use the standard series expansion
\begin{equation}\label{hypergeo}
    {}_2F_1(a,b;c;z)
        = \sum_{n=0}^\infty
          \frac{(a)_n (b)_n}{(c)_n\, n!}\, z^n,
    \qquad |z|<1,
\end{equation}
where $(e)_n$ denotes the Pochhammer symbol:
\begin{equation*}
    (e)_n =
    \begin{cases}
        1, & n=0, \\
        e(e+1)\cdots(e+n-1), & n>0.
    \end{cases}
\end{equation*}
Substituting this expansion yields
\begin{align*}
    H^{-1/2}(z,z')
    &\sim
      \sum_{j=0}^\infty
      \sum_{n=0}^\infty
      \mathcal{A}_{n,j}(r,y,y')
      \, r'^{-d/2 - \mu_j - 2n},
\end{align*}
where the coefficients are given explicitly by
\begin{equation*}
     \mathcal{A}_{n,j}(r,y,y')
     = \frac{\Gamma(1/2+\mu_j)\Gamma(1/2)}{2\Gamma(\mu_j+1)}
       \frac{(1/2+\mu_j)_n (1/2)_n}{(\mu_j+1)_n\, n!}
       \bigl( r^{\,1-\frac d2 + \mu_j + 2n}\, u_j(y) \bigr)
       \overline{u_j(y')}.
\end{equation*}
In particular, the leading term $\mathcal{A}_{0,j}$ exhibits 
$H$-harmonicity.  
Indeed, using the eigenfunction relation
\[
    \left(\Delta_Y + V_0 + \left(\frac{d-2}{2}\right)^2 \right) u_j = \mu_j^2 u_j,
\]
a direct computation shows that for each $j\ge 0$,
\begin{align*}
    H\!\left( r^{1-\frac d2 + \mu_j}\, u_j(y) \right)
    &= \left(
        -\partial_r^2
        - \frac{d-1}{r}\partial_r
        + \frac{\Delta_Y + V_0}{r^2}
      \right)
      \left(
        r^{1-\frac d2 + \mu_j}\, u_j(y)
      \right) =0.
\end{align*}
Thus, the leading coefficient in the rbz expansion is $H$-harmonic. This is precisely the concrete manifestation of the harmonic annihilation 
phenomenon anticipated in Subsection~\ref{sec1.3}, in which the potentially 
problematic leading term is annihilated by the operator $H$, and only higher-order terms contribute to the asymptotic behavior of 
$H_z H^{-1/2}(z,z')$ near the right boundary.

For completeness, we record here the analogue of the preceding calculation
for the left boundary, namely the $\mathrm{lbz}$-face of the doubly blow-up 
space $M_{b,sc}^2$. The analysis is parallel to the rbz-case, with the roles 
of the variables $z=(r,y)$ and $z'=(r',y')$ interchanged. This computation 
provides the full leading behavior of the Riesz potential $H^{-1/2}(z,z')$ 
as either variable approaches the boundary of the compactification 
$\overline{M}$.

We begin again from the resolvent representation.  
For $4 r' \le r$, the formula \eqref{formula2} yields
\begin{equation*}
    (H+\lambda^2)^{-1}(z,z')
    = (rr')^{1-d/2}
      \sum_{j\ge 0} u_j(y)\, \overline{u_j(y')}
      I_{\mu_j}(\lambda r')\, K_{\mu_j}(\lambda r). 
\end{equation*}
Integrating in $\lambda$ and using the same Bessel integral identity as in the 
rbz analysis, we obtain
\begin{align*}
    H^{-1/2}(z,z')\sim \sum_{j\ge 0}
         (rr')^{1-d/2}
         u_j(y)\, \overline{u_j(y')}
         C_j
         \left(\frac{r'}{r}\right)^{\mu_j}
         r^{-1}\,
         {}_2F_1\!\left(
            \frac12+\mu_j,\frac12;\mu_j+1;\frac{r'^2}{r^2}
         \right).
\end{align*}
We expand the hypergeometric function by its standard series \eqref{hypergeo}, yielding the asymptotic expansion at the $\mathrm{lbz}$-face:
\begin{align*}
    H^{-1/2}(z,z')
    &\sim 
     \sum_{j=0}^\infty
     \sum_{n=0}^\infty
     \widetilde{\mathcal{A}}_{n,j}(r',y',y)
     \, r^{-d/2 - \mu_j - 2n},
\end{align*}
where the coefficients are explicitly given by
\begin{equation*}
    \widetilde{\mathcal{A}}_{n,j}(r',y',y)
     = \frac{\Gamma(1/2+\mu_j)\Gamma(1/2)}
            {2\Gamma(\mu_j+1)}
       \frac{(1/2+\mu_j)_n (1/2)_n}
            {(\mu_j+1)_n\, n!}
       \Bigl(
         r'^{\,1-\frac d2 + \mu_j + 2n}\,
         \overline{u_j(y')}
       \Bigr)
       u_j(y).
\end{equation*}
The $H$-harmonicity of the leading term is again transparent.  
Indeed, for each $j\ge 0$,
\begin{align*}
    H_{z'}\!\left( r'^{\,1-\frac d2 + \mu_j}\, \overline{u_j(y')} \right)
    &= \left(
        -\partial_{r'}^2
        - \frac{d-1}{r'}\partial_{r'}
        + \frac{\Delta_{Y'} + V_0}{r'^2}
       \right)
       \left(
         r'^{\,1-\frac d2 + \mu_j}\, \overline{u_j(y')}
       \right) =0
\end{align*}
Thus the asymptotic expansion at the left boundary also exhibits the 
\emph{harmonic annihilation} phenomenon: the leading term is annihilated 
by the operator $H$, and only higher-order terms contribute to 
$H^{1/2}$ near the $\mathrm{lbz}$-face.

\medskip

Combining the expansions at the $\mathrm{rbz}$- and $\mathrm{lbz}$-faces 
provides a precise description of the behavior of the kernel 
$H^{-1/2}(z,z')$ near the boundary of the compactified double space. 
In particular, this illustrates concretely the general mechanism by which 
the potentially problematic leading terms vanish under the action of $H$, 
as predicted by the harmonic annihilation heuristic in 
Subsection~\ref{sec1.3}.

\bigskip
{\bf Acknowledgments.} 
Part of this note is contained in the author's Ph.D thesis \cite{He_phdthesis}. He would like to thank his Ph.D supervisor Adam Sikora for helpful comments and support. He also would like to thank Professor Andrew Hassell for the useful discussion about the parametrix construction and encouragement.

%USE MATHSCI BIBTEX FOR BIBLIORGRAPHY

%\bibliographystyle{apacite}
%\bibliographystyle{acm}
\bibliographystyle{abbrv}

\bibliography{bibl.bib}

\end{document}